\documentclass[11pt,a4paper]{amsart}

\title[The $\debar$-equation, duality, and holomorphic forms]{The $\debar$-equation, duality,
and holomorphic forms on a reduced complex space}

\author{H{\aa}kan Samuelsson Kalm}
\thanks{The author was partially supported by the Swedish Research Council.}
\address{H{\aa}kan Samuelsson Kalm, Department of Mathematical Sciences, Division of Mathematics, University of Gothenburg and 
Chalmers University of Technology, SE-412 96 G\"{o}teborg, Sweden}
\email{hasam@chalmers.se}

\date{\today}

\usepackage[english]{babel}
\usepackage{amsmath,calligra}
\usepackage{amsthm,amssymb,latexsym}
\usepackage[all,cmtip]{xy}
\usepackage{url,ae}
\usepackage{t1enc}
\usepackage{mathrsfs}
\usepackage{graphicx}
\usepackage{geometry}
\newtheorem{proposition}{Proposition}[section]
\newtheorem{theorem}[proposition]{Theorem}
\newtheorem{lemma}[proposition]{Lemma}
\newtheorem{corollary}[proposition]{Corollary}

\theoremstyle{definition}

\newtheorem{example}[proposition]{Example}
\newtheorem{remark}[proposition]{Remark}

\numberwithin{equation}{section}

\DeclareMathOperator{\Hom}{\mathscr{H}\text{\kern -3pt {\calligra\Large om}}\,}
\DeclareMathOperator{\Ext}{\mathscr{E}\text{\kern -3pt {\calligra\Large xt}}\,\,}
\DeclareMathOperator{\Image}{\mathscr{I}\text{\kern -3pt {\calligra\Large m}}\,}
\DeclareMathOperator{\Kernel}{\mathscr{K}\text{\kern -3pt {\calligra\Large er}}\,}
\newcommand{\C}{\mathbb{C}}
\newcommand{\debar}{\bar{\partial}}
\newcommand{\A}{\mathscr{A}}
\newcommand{\Bsheaf}{\mathscr{B}}
\newcommand{\F}{\mathscr{F}}
\newcommand{\G}{\mathscr{G}}
\newcommand{\HH}{\mathscr{H}}
\newcommand{\R}{\mathbb{R}}
\newcommand{\J}{\mathcal{J}}
\newcommand{\E}{\mathscr{E}}

\newcommand{\W}{\mathcal{W}}
\newcommand{\PM}{\mathcal{PM}}

\newcommand{\hol}{\mathscr{O}}

\newcommand{\K}{\mathscr{K}}

\newcommand{\Proj}{\mathscr{P}}

\newcommand{\CH}{\mathscr{C} \kern -2pt \mathscr{H}}
\newcommand{\B}{\mathbb{B}}
\newcommand{\Om}{\mathit{\widehat{\Omega}}}
\newcommand{\ett}{\mathbf{1}}

\def\newop#1{\expandafter\def\csname #1\endcsname{\mathop{\rm #1}\nolimits}}
\newop{span}

\begin{document}
\nocite{*}
\bibliographystyle{plain}

\begin{abstract}
We solve the $\debar$-equation for $(p,q)$-forms locally on any 
reduced pure-dimensional complex space and we prove an explicit version of Serre duality
by introducing suitable concrete fine sheaves of certain $(p,q)$-currents. In particular
this gives a precise condition for the $\debar$-equation to be globally solvable. 
Our results extend results for $(0,q)$-forms and give information about
holomorphic $p$-forms on singular spaces. 
%  

% We study two natural notions of holomorphic forms on a reduced
% pure $n$-dimensional complex space $X$: sections of the sheaves $\Om_X^{\bullet}$ of germs of
% holomorphic forms on $X_{reg}$ that have a holomorphic extension  
% to some ambient complex manifold, and sections of the sheaves $\omega_X^{\bullet}$ introduced 
% by Barlet. We show that $\Om_X^p$ and $\omega_X^{n-p}$ are Serre dual to each other
% in a certain sense. We also provide explicit, intrinsic and semi-global Koppelman formulas 
% for the $\debar$-equation on $X$ and introduce fine sheaves $\mathscr{A}_X^{p,q}$
% and $\mathscr{B}_X^{p,q}$ of $(p,q)$-currents on $X$, that are smooth on $X_{reg}$,
% such that $(\mathscr{A}_X^{p,\bullet},\debar)$ is a resolution of $\Om_X^p$
% and, if $\Om_X^{n-p}$ is Cohen-Macaulay, $(\mathscr{B}_X^{p,\bullet},\debar)$
% is a resolution of $\omega_X^{p}$.
\end{abstract}

\maketitle
\thispagestyle{empty}

\section{Introduction}
Let $X$ be a reduced complex space of pure dimension $n$. A smooth form on $X$ is locally the pullback to 
$X_{reg}$ of a smooth form in some ambient complex manifold; it is well-known that this is an intrinsic notion
and we denote the corresponding sheaf by $\E_X$.
It is proved in \cite{AS} that if $\varphi$ is a smooth $\debar$-closed $(0,q)$-form, $q>0$, on $X$ and $X$ is Stein,
then there is a smooth $(0,q-1)$-form $\psi$ on $X_{reg}$ such that $\debar\psi=\varphi$; if $q=0$ then $\varphi$ is 
strongly holomorphic. In general $\psi$ cannot be smooth on $X$, see, e.g., \cite[Example~1.1]{AS}, but the local solution operators
constructed in \cite{ASJFA} and \cite{AS} provide solutions $\psi$ with certain mild singularities at $X_{sing}$. In particular $\psi$ 
is a current on $X$ and $\debar\psi=\varphi$ in the current sense also across $X_{sing}$.

In case $X$ is smooth, local existence results for the $\debar$-equation for $(0,q)$-forms easily carry over to
$(p,q)$-forms since the holomorphic $p$-forms in this case are sections of a vector bundle, i.e., a locally free sheaf,
denoted $\mathit{\Omega}^p_X$. 
In the presence of singularities the situation is quite different. There are several natural notions of 
holomorphic $p$-forms and usually the corresponding sheaves are not locally free. We will be particularly interested in
two notions, $\Om_X^p$ and $\omega_X^p$. The sheaf $\omega_X^p$ was introduced by Barlet \cite{Barlet} and 
$\Om_X^p$ can be defined in the same way as $\E_X$ above replacing ``smooth form'' by ``holomorphic $p$-form''.
It is well-known that $\Om_X^p$
is the K\"{a}hler-Grothendieck $p$-forms modulo torsion, see, e.g., \cite{Ferrari} or Section~\ref{strong} below. It is clear that
$\mathit{\widehat{\Omega}}_X^0=\hol_X$ and well-known that $\omega_X^n$ is the Grothendieck dualizing sheaf.

Our main result is that locally on $X$ the $\debar$-equation for $(p,q)$-forms is always solvable,
if interpreted in the sense of currents even across $X_{sing}$.
For other results about the $\debar$-equation in the singular setting see, e.g.,
\cite{BS}, \cite{FOV}, \cite{LR}, \cite{Ohsawa}, \cite{OV}, \cite{PS}, \cite{JeanDuke}, \cite{Saper}.
Recall that the $(p,q)$-currents on $X$ is the dual of the compactly supported sections of $\E_X^{n-p,n-q}$; given an embedding 
$X\hookrightarrow M$ they can also be identified with certain currents in ambient space, see Section~\ref{prelim}.

\begin{theorem}\label{main2}
Let $X$ be a pure $n$-dimensional analytic subset of a pseudoconvex domain $D\subset \C^N$,
let $D'\Subset D$ and set $X':=X\cap D'$. There are integral operators $\K\colon\mathscr{E}^{p,q}(X)\to\mathscr{E}^{p,q-1}(X'_{reg})$
and $\Proj\colon\mathscr{E}^{p,0}(X)\to\Om^p(X')$ such that $\K\varphi$ has a current extension to $X'$ and, 
as currents on $X'$, 
\begin{eqnarray*}
\varphi &=& \K(\debar\varphi) + \Proj\varphi, \quad \varphi\in\mathscr{E}^{p,0}(X), \\
\varphi &=& \debar \K\varphi + \K(\debar\varphi), \quad \varphi\in\mathscr{E}^{p,q}(X),\,\, q\geq 1.
\end{eqnarray*}
\end{theorem}

The construction of $\Proj$ shows that $\Proj\varphi$ in fact has a holomorphic extension to
$D'$. 
The integral operators $\K$ and $\Proj$ are given by kernels $k(\zeta,z)$ and $p(\zeta,z)$
which are currents on $X\times X'$ that
are respectively integrable and smooth on $X_{reg}\times X'_{reg}$  
and that have principal value-type singularities at the singular locus of $X\times X'$. 
Since a current locally has finite order we get the following result.

\begin{corollary}\label{malgrangecor}
Let $\varphi$ be a smooth $\debar$-closed $(p,q)$-form on $X_{reg}$ such that 
there is a $C^{\ell}$-smooth form in $D$ whose pullback to $X_{reg}$ equals $\varphi$.
There is an $M_{D'}\geq 0$, independent of $\varphi$, such that the following holds.
\begin{itemize}
\item[(i)] If $q=0$ and $\ell\geq M_{D'}$ then there is a $\tilde{\varphi}\in\Om^p(X')$ 
such that $\varphi_{\restriction_{X'_{reg}}}=\tilde{\varphi}_{\restriction_{X'_{reg}}}$.

\item[(ii)] If $q\geq 1$ and $\ell\geq M_{D'}$ then there is a smooth $(p,q-1)$-form $u$ on $X'_{reg}$
such that $\debar u=\varphi$ on $X'_{reg}$.
\end{itemize}
\end{corollary}

Part (i) for $p=0$ and $M_{D'}=\infty$ is a classical result by Malgrange \cite[Th\'{e}or\`{e}me~4]{Malgrange}
answering a question by Grauert; for $M_{D'}<\infty$ it is due to Spallek \cite{Spallek65}. 
Part (ii) for $p=0$ and $X$ a reduced complete intersection was first proved by
Henkin and Polyakov \cite{HePo}.
For $p=0$, Corollary~\ref{malgrangecor} is also proved in \cite{ASJFA}. 
We remark that Corollary~\ref{malgrangecor} is explicit in the 
sense that $\Proj \varphi$ (resp.\ $\K\varphi$) provides an explicit holomorphic extension of $\varphi$ to $D'$
(resp.\ explicit solution to $\debar u=\varphi$ on $X'_{reg}$).

%%%%%%%%%%%%%%%%

As already mentioned, $\debar\psi=\varphi$ is in general not smoothly solvable in neighborhoods of singular points
even if $\varphi$ is smooth (and $\debar$-closed), i.e., the complex $(\E_X^{p,\bullet},\debar)$ is in general not exact. 
Therefore $\K\varphi$ cannot be smooth in general. However, the singularities of $\K\varphi$ are not worse than that one can apply 
another $\K$-operator. In fact, one can apply $\K$-operators repeatedly. Using this one can construct, see Section~\ref{Asubsection} below
for details, sheaves $\A_X^{p,q}$ of certain currents, which are
closed under $\K$-operators and $\debar$. 
We have the following generalization of \cite[Theorem~1.2]{AS}.

\begin{theorem}\label{main1}
Let $X$ be a reduced complex space of pure dimension $n$. For each $p=0,\ldots,n$ 
there are fine sheaves $\A_X^{p,q}$, $q=0,\ldots,n$, of $(p,q)$-currents on $X$ with the 
standard extension property such that
\begin{itemize}
\item[(i)] $\E_X^{p,q}\subset \A_X^{p,q}$ and $\oplus_q \A_X^{p,q}$ is a module over
$\oplus_q\E_X^{0,q}$,

\item[(ii)] $\A_{X_{reg}}^{p,q}=\E_{X_{reg}}^{p,q}$,

\item[(iii)] the following sheaf complex is exact
\begin{equation}\label{intro:Akplx}
0\to \Om_X^p \hookrightarrow \A_X^{p,0} \stackrel{\debar}{\longrightarrow} \A_X^{p,1} \stackrel{\debar}{\longrightarrow}
\cdots \stackrel{\debar}{\longrightarrow} \A_X^{p,n} \to 0.
\end{equation}
\end{itemize}
\end{theorem}

That a current has the standard extension property (SEP) means roughly speaking that it is determined by its
restriction to any dense Zariski open subset, see Section~\ref{prelim} for the precise definition.

Since the $\A_X$-sheaves are fine, the de Rham theorem gives
the following generalization to the singular setting of the classical Dolbeault isomorphism.

\begin{corollary}\label{cor1}
Let $X$ be a reduced complex space of pure dimension, let $F\to X$ be a holomorphic vector bundle,
and let $\F$ be the associated locally free $\hol_X$-module. Then
\begin{equation*}
H^q(X,\F\otimes\Om_X^p) \simeq H^q\big(\A^{p,\bullet}(X,F),\debar\big).
\end{equation*}
\end{corollary}

Notice that since $(\A_X^{p,\bullet},\debar)$ is a resolution of $\Om_X^p$, whose sections in particular are smooth, and 
since $\E_X^{p,\bullet}\subset \A_X^{p,\bullet}$,
it follows from a well-known construction that each cohomology class in $H^q\big(\A^{p,\bullet}(X),\debar\big)$ 
has a smooth representative (cf., e.g., \cite[Section~7]{RSWSerre}).

The operators $\K$ and $\Proj$ in Theorem~\ref{main2} extend to operators $\A^{p,q}(X)\to\A^{p,q-1}(X')$
and $\A^{p,0}(X)\to \Om^p(X')$, respectively, and the integral formulas continue to hold; it is this generalized
version of Theorem~\ref{main2} that we will prove below. The operators $\K$ and $\Proj$ can be applied to more general
currents, for instance to any semi-meromorphic current. However, the integral formulas of Theorem~\ref{main2} then 
cannot hold in general. Indeed, if this were the case then, in particular, any $\debar$-closed meromorphic $p$-form
on $X$ would be in $\Om^p(X')$. This is to say that $\omega_X^p=\Om^p_X$, which is not true in general.
On the other hand, the obstruction to the integral formulas to hold is explicit and gives a residue criterion,
formulated in Theorem~\ref{rescrit} below, for a 
meromorphic $p$-form to be a section of $\Om^p_X$. This generalizes results by Tsikh \cite{Tsikh}, Andersson \cite{Acrit},
and Henkin-Passare \cite{HP}. The residue criterion leads to a geometric criterion, Proposition~\ref{extensionprop}, which in turn gives
the following geometric characterization of complex spaces with the property that any holomorphic $p$-form
on the regular part extends to a section of $\Om^p_X$. Recall that to a coherent analytic sheaf $\mathscr{G}$ on $X$ there 
are associated \emph{singularity subvarieties} $S_{0}(\mathscr{G})\subset S_1(\mathscr{G})\subset \cdots \subset X$, see, e.g., 
\cite[\S 1]{SiuTraut} or Section~\ref{AWcurrsektion} below.

\begin{proposition}\label{serreconditions}
Let $X$ be a reduced complex space of pure dimension $n$. Then the following conditions are equivalent.
\begin{itemize}
\item[(i)] $\textrm{codim}_X X_{sing} \geq 2$
and $\textrm{codim}_X S_{n-k}(\Om_X^p) \geq k+2$ for $k\geq 1$.

\item[(ii)] For any open $U\subset X$ the restriction map $\Om^p(U)\to \Om^p(U_{reg})$ is bijective.
\end{itemize}
\end{proposition}
 
This result is a variation on \cite[Theorem~1.14]{SiuTraut}, see also \cite{Scheja}, that is explicit 
in the sense mentioned above. Notice that for $p=0$ Proposition~\ref{serreconditions} is a normality criterion.
It is however also possible to verify directly that condition (i) with $p=0$ is equivalent to Serre's
conditions $R1$ and $S2$. From Proposition~\ref{serreconditions} we get the following result, see 
the end of Section~\ref{Asubsection}

\begin{corollary}\label{smoothcor}
Assume that $X$ is a reduced complete intersection. Then $X$ is smooth if and only if condition (i)
of Proposition~\ref{serreconditions} with $p=n$ holds.
\end{corollary}

\smallskip

In view of Corollary~\ref{cor1}, $H^q(X,\Om_X^p)$ encodes the global obstructions to solving the $\debar$-equation.
To get some control of these obstructions we will describe the 
dual of $H^q(X,\Om_X^p)$ in a way similar to the description given in Corollary~\ref{cor1} of $H^q(X,\Om_X^p)$ itself. 
This dual description provides a concrete analytic realization of 
Serre duality in the singular setting analogous to the classical one in the non-singular case. To do this we introduce
``dual'' concrete fine sheaves $\Bsheaf_X^{n-p,n-q}$ of certain $(n-p,n-q)$-currents on $X$.
The operators $\K$ and $\Proj$ correspond to integrating the kernels $k(\zeta,z)$ and $p(\zeta,z)$, respectively,
with respect to $\zeta$. Integrating with respect to $z$ instead gives operators $\check{\K}$ and $\check{\Proj}$
with different properties and applying $\check{\K}$-operators repeatedly gives these dual $\Bsheaf_X$-sheaves.
The case $p=n$ of the following result is proved in \cite{RSWSerre}. 

\begin{theorem}\label{main3}
Let $X$ be a reduced complex space of pure dimension $n$. For each $p=0,\ldots,n$ there are fine
sheaves $\Bsheaf_X^{p,q}$, $q=0,\ldots,n$, of $(p,q)$-currents on $X$ with the SEP such that
\begin{itemize}
\item[(i)] $\E_X^{p,q}\subset \Bsheaf_X^{p,q}$ and $\oplus_q\Bsheaf_X^{p,q}$ is a module over $\oplus_q\E_X^{0,q}$,
\item[(ii)] $\Bsheaf_{X_{reg}}^{p,q} = \E_{X_{reg}}^{p,q}$,
\item[(iii)] $0\to \Bsheaf_X^{p,0} \stackrel{\debar}{\longrightarrow} \Bsheaf_X^{p,1}
\stackrel{\debar}{\longrightarrow} \cdots \stackrel{\debar}{\longrightarrow} \Bsheaf_X^{p,n} \to 0$
is a sheaf complex with coherent cohomology sheaves $\omega_X^{p,q}:=\HH^q(\Bsheaf_X^{p,\bullet},\debar)$
and $\omega_X^p=\omega_X^{p,0}$. If $\Om_X^{n-p}$ is Cohen-Macaulay then $(\Bsheaf_X^{p,\bullet},\debar)$
is a resolution of $\omega_X^p$.
\end{itemize}
\end{theorem}

The proof of Theorem~\ref{main3} will show that if $i\colon X\hookrightarrow D\subset \C^N$, then
$i_*\omega_X^{p,q}\simeq \Ext_{\hol}^{\kappa+q}(\Om_X^{n-p},\mathit{\Omega}^N)$, where $\kappa=N-n$, $\hol=\hol_{\C^N}$,
and $\mathit{\Omega}^N=\mathit{\Omega}^N_{\C^N}$; we will use this notation throughout.

\begin{theorem}\label{main4}
Let $X$ be a pure $n$-dimensional analytic subset of a pseudoconvex domain $D\subset \C^N$,
let $D'\Subset D$ and set $X':=X\cap D'$. There are integral operators 
$\check{\K}\colon \Bsheaf^{p,q}(X)\to \Bsheaf^{p,q-1}(X')$ and 
$\check{\Proj}\colon \Bsheaf^{p,q}(X)\to \Bsheaf^{p,q}(X')$ such that
\begin{equation*}
\psi=\debar\check{\K}\psi + \check{\K}(\debar\psi) + \check{\Proj}\psi
\end{equation*}
on $X'$. If $\Om_X^{n-p}$ is Cohen-Macaulay and $\psi\in\Bsheaf^{p,q}(X)$ then
$\check{\Proj}\psi\in \omega^{p}(X')$ if $q=0$ and $\check{\Proj}\psi=0$ if $q\geq 1$.
% \begin{itemize}
% \item[(i)] $\psi=\debar\check{\K}\psi + \check{\K}(\debar\psi) + \check{\Proj}\psi$ in the current sense on $X'$,
% \item[(ii)] $\check{\Proj}\big(\Bsheaf^{p,q}(X)\big) \subset \textrm{Ker} \big(\debar \colon\Bsheaf^{p,q}(X)\to \Bsheaf^{p,q-1}(X')\big)$.
% \end{itemize}
\end{theorem}

Notice that if $\psi\in \omega^p(X)$ then, on $X'$, $\psi=\check{\Proj}\psi$ is a representation formula
for sections of $\omega_X^p$.

If $\varphi\in\A^{p,q}(X)$ and $\psi\in\Bsheaf^{n-p,n-q}(X)$ then the product $\varphi\wedge\psi$ exists, Theorem~\ref{RSWthm}.
On $X_{reg}$ it is just the exterior product of smooth forms, and this form turns out to have a unique extension to $X$
as a current with the SEP. Moreover, $\debar(\varphi\wedge\psi)=\debar\varphi\wedge\psi+(-1)^{p+q}\varphi\wedge\debar\psi$.
Hence, there is a pairing, the trace map, $(\varphi,\psi)\mapsto\int_X\varphi\wedge\psi$ and it descends to 
a trace map on cohomology.

\begin{theorem}\label{main5}
Let $X$ be a compact reduced complex space of pure dimension $n$ and let $F\to X$
be a holomorphic vector bundle. 
%If $H^{q}(X,\F\otimes\Om_X^p)$ and $H^{q+1}(X,\F\otimes\Om_X^p)$ are Hausdorff, then the pairing
Then the following pairing is non-degenerate
\begin{equation}\label{par1}
H^q\big(\A^{p,\bullet}(X,F),\debar\big) \times H^{n-q}\big(\Bsheaf^{n-p,\bullet}(X,F^*),\debar\big) \to \C,
\end{equation}
\begin{equation}\label{par2}
([\varphi]_{\debar},[\psi]_{\debar}) \mapsto \int_X \varphi\wedge\psi.
\end{equation}
\end{theorem}

The case $p=0$ is proved in \cite{RSWSerre}.
It follows from Theorem~\ref{main5} together with Corollary~\ref{cor1} and Theorem~\ref{main3} that if $\Om_X^p$ is Cohen-Macaulay, then
there is a non-degenerate pairing $H^q(X,\Om_X^p)\times H^{n-q}(X,\omega_X^{n-p})\to \C$. For $p=0$ this is the 
well-known duality on Cohen-Macaulay spaces. 
For $p>0$ it follows that Barlet's sheaf $\omega_X^{n-p}$ is dualizing with respect to $\Om_X^p$ in the same way 
as the Grothendieck sheaf $\omega_X^n$ is dualizing with respect to $\hol_X$.
If $\Om_X^p$ is not Cohen-Macaulay,
then $\omega_X^{n-p}$ does not suffice to describe the dual of $H^q(X,\Om_X^p)$;
higher $\Ext$'s come into play. This is also the case in the classical duality by Ramis and Ruget \cite{RaRu}:
Given a coherent sheaf $\F$ on $X$ they describe the dual of
$H^q(X,\F)$ as $\text{Ext}^{-q}(X;\F,\mathbf{K}_X^{\bullet})$, where $\mathbf{K}_X^{\bullet}$ is 
the dualizing complex in the sense of \cite{RaRu}.

We notice the following consequence of Theorem~\ref{main5}: 
If $\varphi$ is a smooth $\debar$-closed $(p,q)$-form on $X$, then there is a smooth solution to the $\debar$-equation
on $X_{reg}$ if $\int_X\varphi\wedge\psi=0$ for all $\debar$-closed smooth $(n-p,n-q)$-forms $\psi$ on $X$.
Indeed, $\varphi$ defines an element in $H^{q}(\Bsheaf^{p,\bullet}(X),\debar)$ and 
each element in $H^{n-q}(\A^{n-p,\bullet}(X),\debar)$ has a smooth representative.

With a slight modification of the statement our Serre duality, Theorem~\ref{main5}, continues to hold on paracompact spaces provided 
certain separability conditions are fulfilled. In fact, 
instead of proving Theorem~\ref{main5}, we will prove the following slightly more general result: 

\emph{If $X$ is a reduced paracompact complex space of pure dimension $n$
and we replace $\Bsheaf^{n-p,\bullet}(X,F^*)$ in Theorem~\ref{main5} 
by the corresponding space of sections with compact support, then 
the conclusion of Theorem~\ref{main5} holds provided that $H^{q}(X,\F\otimes\Om_X^p)$ 
and $H^{q+1}(X,\F\otimes\Om_X^p)$ are Hausdorff.}

We remark that the Hausdorff assumption is automatically fulfilled if 
$X$ is compact or holomorphically convex; this follows from the Cartan-Serre theorem and Prill's result, \cite{Prill},
respectively.
Moreover, by the Andreotti-Grauert theorem, 
$H^{q}(X,\F\otimes\Om_X^p)$ and $H^{q+1}(X,\F\otimes\Om_X^p)$ are Hausdorff for $q\geq k$ if $X$ is $k$-convex.

\smallskip

{\bf Acknowledgment:} I would like to thank Professor Daniel Barlet for important comments on an 
earlier version of this paper as well as for finding and letting us include the alternative 
proof of Proposition~\ref{BHPprop} below.

%%%%%%%%%%%%%%%%%%%%%%%%%%%%%%%%%%%%%%%%%%%%%%%%%%%%%%%%%%%%%%%%%%%%%%%%%%%%

\section{Preliminaries}\label{prelim}
Let $X$ be a pure $n$-dimensional reduced complex space. Following \cite[Section~4.2]{HL},
the $(p,q)$-currents on $X$ is the dual of the $(n-p,n-q)$-test forms $\mathscr{D}^{n-p,n-q}(X)$, i.e.,
the compactly supported sections of $\E_X^{n-p,n-q}$. 
More concretely, if $i\colon X\hookrightarrow D\subset\C^N$ is an embedding and $\mu$ is a
$(p,q)$-current on $X$, then $\nu:=i_*\mu$ is a $(p+\kappa,q+\kappa)$-current in $D$ (recall that $\kappa:=N-n$) 
and $\nu . \xi=0$ for any test form $\xi$ in $D$ whose pullback to $X_{reg}$ vanishes. Conversely,
if $\nu$ is such a current in $D$ then there is a current $\mu$ on $X$ such that $\nu=i_*\mu$.

Let $\chi$ be any smooth regularization of the characteristic function of $[1,\infty)\subset \R$;
throughout the paper, $\chi$ will denote such a function.
A current $\mu$ on $X$ is said to have the \emph{standard extension property} (SEP) with respect to
a subvariety $Z\subset X$ if 
$\chi(|h|^2/\epsilon)\mu_{\restriction_U} \to \mu_{\restriction_U}$ as $\epsilon \to 0$ for any open $U\subset X$,
where $h$ is any holomorphic tuple on $U$ not vanishing identically on any irreducible component of $Z\cap U$.
If $Z=X$ we simply say that $\mu$ has the SEP (on $X$).,

\subsection{Meromorphic forms}
Let here $X$ be a pure-dimensional analytic subset of some domain $D\subset \C^N$ and let $W$
be an analytic subset containing $X_{sing}$ but not any irreducible component of $X$.
It is proved in \cite{HP} that the following conditions on a holomorphic $p$-form $\varphi$ on $X\setminus W$
are equivalent. 1) $\varphi$ is locally the pullback to $X\setminus W$ of a meromorphic $p$-form in
a neighborhood of $X$. 2) For any desingularization $\pi\colon\tilde{X}\to X$ such that
$\pi^{-1}X_{reg}\simeq X_{reg}$, $\pi^*\varphi$ has a meromorphic extension to $\tilde{X}$.
3) There is a current $T$ in $D$ such that $i_*\varphi =T_{\restriction_{D\setminus W}}$, where 
$i\colon X\hookrightarrow D$ is the inclusion. 4) For any $h\in\hol(X)$ that vanishes on $W$, but
not identically on any component of $X$, the current
\begin{equation}\label{pool}
\mathscr{D}^{n-p,n}(X)\ni\xi\mapsto \lim_{\epsilon\to 0}\int_X\chi(|h|^2/\epsilon) \varphi\wedge\xi
\end{equation}
exists and is independent of $h$. 

The sheaf of germs of $p$-forms satisfying these conditions is called the sheaf of germs of
\emph{meromorphic $p$-forms} on $X$; we will denote it by $\mathscr{M}^p_X$. One can check 
that if $x\in X$ is an irreducible point then $\mathscr{M}^0_{X,x}$ is (isomorphic to)
the field of fractions of $\hol_{X,x}$. We usually make no distinction
between a meromorphic form $\varphi$ and the associated principal value current \eqref{pool}.

\subsection{Pseudomeromorphic currents}\label{PMsektion}
Pseudomeromorphic currents were introduced in \cite{AW2}; the definition we 
need and will use is from \cite{AS}.
In one complex variable $z$ it is elementary to see that the principal value current $1/z^m$ exists
and can be defined, e.g., as the limit as $\epsilon \to 0$
of $\chi(|h(z)|^2/\epsilon)/z^m$, where $h$ is a holomorphic function (or tuple) vanishing at $z=0$, 
or as the value at $\lambda=0$ of the analytic continuation
of the current-valued function $\lambda \mapsto |h(z)|^{{2\lambda}}/z^m$. 
It follows that the \emph{residue current} $\debar(1/z^m)$ can be computed as the 
limit of $\debar\chi(|h(z)|^2/\epsilon)/z^m$ or as the value at $\lambda=0$ of 
$\lambda \mapsto \debar |h(z)|^{{2\lambda}}/z^m$.
Since tensor products of currents are well-defined we can form the current
\begin{equation}\label{elementary}
\tau=\debar \frac{1}{z_1^{m_1}}\wedge \cdots \wedge \debar \frac{1}{z_r^{m_r}}\wedge 
\frac{\gamma(z)}{z_{r+1}^{m_{r+1}}\cdots z_n^{m_n}}
\end{equation}
in $\C^n$, where $m_1,\ldots,m_r$ are positive integers, $m_{r+1},\ldots,m_{n}$ are nonnegative integers, and 
$\gamma$ is a smooth compactly supported form. Notice that $\tau$ is anti-commuting in the residue factors 
$\debar(1/z_j^{m_j})$ and commuting in the principal value factors $1/z_k^{m_k}$.
We say that a current of the form \eqref{elementary} is an \emph{elementary pseudomeromorphic current}. 
Let $X$ be a pure-dimensional reduced complex space and let $x\in X$. 
We say that a germ of a current $\mu$ at $x$ is {\em pseudomeromorphic}
if it is a finite sum of pushforwards $\pi_*\tau=\pi_*^1 \cdots \pi_*^{\ell} \tau$, where $\mathcal{U}$ is 
a neighborhood of $x$,
\begin{equation*}
\mathcal{U}^{\ell} \stackrel{\pi^{\ell}}{\longrightarrow} \cdots \stackrel{\pi^{2}}{\longrightarrow} \mathcal{U}^1
\stackrel{\pi^{1}}{\longrightarrow} \mathcal{U}^0=\mathcal{U},
\end{equation*}
each $\pi^j$ is either a modification, a simple projection 
$\mathcal{U}^j=\mathcal{U}^{j-1}\times Z\to \mathcal{U}^{j-1}$, or 
an open inclusion, and $\tau$ is an elementary pseudomeromorphic current on $\mathcal{U}^{\ell}\subset\C^N$. 
The union of all germs of pseudomeromorphic currents on $X$ forms an open subset of the sheaf
of germs of currents on $X$ and thus defines a subsheaf $\PM_X$. Notice that since $\debar$ maps an elementary
pseudomeromorphic current to a sum of such currents it follows that $\debar$ maps $\PM_X$ to itself.

The following result is fundamental and will be used repeatedly in this paper.

\medskip

\noindent {\bf Dimension principle.} \emph{Let $X$ be a reduced pure-dimensional
complex space, let $\mu\in\PM(X)$, and assume that $\mu$ has support contained in a subvariety $V\subset X$. 
If $\mu$ has bidegree $(*,q)$ and $\textrm{codim}_X V>q$, then $\mu=0$.}

\medskip

This result is from \cite{AW2}, see also \cite[Proposition~2.3]{AS}. In connection to the dimension principle
we also mention that if $\mu\in\PM(X)$, $\textrm{supp}\,\mu\subset V$, and $h$ is a holomorphic function 
vanishing on $V$, then $\bar{h}\mu=0$ and $d\bar{h}\wedge\mu=0$. 
An arbitrary current $\mu$ with $\textrm{supp}\,\mu\subset V$ is of the form $\mu=i_*\tau$, where 
$i$ is the inclusion of $V$, for some current $\tau$ on $V$ if and only if
$h\mu=dh\wedge\mu=\bar{h}\mu=d\bar{h}\wedge\mu=0$ for all holomorphic $h$ vanishing on $V$.
Thus, if $\mu\in\PM(X)$, there is such a $\tau$
if and only if $h\mu=dh\wedge\mu=0$ for all holomorphic functions $h$
vanishing on $V$.

Another fundamental property of pseudomeromorphic currents is that they can be ``restricted''
to analytic (or constructible) subsets: Let $\mu\in\PM(X)$, let $V\subset X$ be an analytic subset,
and set $V^c:=X\setminus V$. Then the restriction of $\mu$ to the open subset $V^c$ has a natural 
pseudomeromorphic extension $\ett_{V^c}\mu$ to $X$. It follows that $\ett_V\mu:=\mu-\ett_{V^c}\mu$
is a pseudomeromorphic current with support contained in $V$. In \cite{AW2} $\ett_{V^c}\mu$ is defined as the 
value at $0$ of the analytic continuation of the current-valued function $\lambda\mapsto |h|^{2\lambda}\mu$,
where $h$ is any holomorphic tuple with zero set $V$; $\ett_{V^c}\mu$ can also be defined as
$\lim_{\epsilon\to 0}\chi(|h|^2v/\epsilon)\mu$, where $v$ is any smooth strictly positive function,
see \cite[Lemma~3.1]{AW3}, cf.\ also \cite[Lemma~6]{LSprodukter}.\footnote{$\epsilon$-approximations and 
$\lambda$-approximations can be used interchangeably; $\lambda$-approximations
are often computationally easier to work with while we believe that $\epsilon$-approximations
are conceptually easier. For the rest of this paper we will work with
$\epsilon$-approximations.} Taking restrictions is commutative, in fact, if $V$ and $W$ are any 
constructible subsets then $\ett_V\ett_W\mu=\ett_{V\cap W}\mu$.
Let us also notice that $\mu\in \PM(X)$ has the SEP (on $X$) precisely means that 
$\ett_V\mu=0$ for all germs of analytic subsets $V\subset X$ of positive codimension. We will denote by $\W_X$
the subsheaf of $\PM_X$ of currents with the SEP on $X$. From \cite[Section~3]{AW3} it follows that if 
$\pi\colon X'\to X$ is either a modification, a simple projection, or an open inclusion, and 
$\mu\in \W(X')$ then $\pi_*\mu\in\W(X)$.

\begin{lemma}\label{seplemma}
Let $X$ be a reduced complex space and let $Y\subset X$ be an analytic nowhere
dense subset. If $\mu\in\PM(X)\cap\W(X\setminus Y)$ then
$\ett_{X\setminus Y}\mu\in\W(X)$.
\end{lemma}

\begin{proof}
Let $V\subset X$ be a germ of an analytic nowhere dense subset. 
Since $\mu\in\W(X\setminus Y)$ we see that $\textrm{supp}\,\ett_V\mu \subset Y\cap V$
and so $\ett_V\ett_{X\setminus Y}\mu=\ett_{X\setminus Y}\ett_V\mu=0$.
\end{proof}

For future reference we give the following simple lemma, part (i) of which is almost tautological.

\begin{lemma}\label{obslemma}
Let $X$ be a germ of a reduced complex space and let $\mu\in\W(X)$.
\begin{itemize}
\item[(i)] We have that $\debar\mu\in\mathcal{W}(X)$ if and only if 
$\lim_{\epsilon\to 0} \debar \chi(|h|^2/\epsilon)\wedge\mu=0$ for all generically
non-vanishing holomorphic tuples $h$ on $X$.

\item[(ii)] Let $Y\subset X$ be an analytic nowhere dense subset,
let $h$ be a holomorphic tuple such that $Y=\{h=0\}$, and assume that 
$\debar\mu\in\W(X\setminus Y)$. Then $\debar\mu\in\W(X)$ if and only if
$\lim_{\epsilon\to 0} \debar \chi(|h|^2/\epsilon)\wedge\mu=0$.
\end{itemize}

\end{lemma}
\begin{proof}
Since $\mu\in \mathcal{W}(X)$ we have that $\mu=\lim_{\epsilon\to 0} \chi(|h|^2/\epsilon)\mu$
for any generically non-vanishing $h$. It follows that
\begin{equation}\label{kungsbacka}
\debar\mu=\lim_{\epsilon\to 0} \debar(\chi(|h|^2/\epsilon)\mu)=
\lim_{\epsilon\to 0} \debar\chi(|h|^2/\epsilon)\wedge\mu + \lim_{\epsilon\to 0} \chi(|h|^2/\epsilon)\debar\mu.
\end{equation}
Now, $\debar\mu\in\mathcal{W}(X)$ if and only if the last term on the right hand side equals $\debar\mu$ 
for all generically non-vanishing $h$ and part (i) of the lemma follows. The ``only if'' part of
(ii) also follows directly from \eqref{kungsbacka}. On the other hand, if
$\lim_{\epsilon\to 0} \debar \chi(|h|^2/\epsilon)\wedge\mu=0$ then, by \eqref{kungsbacka},
$\debar\mu=\ett_{X\setminus Y}\debar\mu$ and so the ``if'' part of (ii) follows from Lemma~\ref{seplemma}.
\end{proof}

Recall that a current on $X$ is said to be semi-meromorphic if it a principal value current of the form
$\alpha/f$, where $\alpha$ is a smooth form and $f$ is a holomorphic function or section 
of a complex line bundle such that $f$ does not vanish identically on any component of $X$.
Following \cite{AS}, see also \cite{AW3}, we say that a current $a$ on $X$ is 
\emph{almost semi-meromorphic} if there is a modification $\pi\colon X'\to X$ and a semi-meromorphic 
current $\alpha/f$ on $X'$ such that $a=\pi_*(\alpha/f)$; if $f$ takes values in $L\to X'$ we need
also $\alpha$ to take values in $L\to X'$ if we want $a$ to be scalar valued.
If $a$ is almost semi-meromorphic on $X$ then the smallest Zariski-closed set outside of which $a$ is smooth
has positive codimension and is denoted $ZSS(a)$, the \emph{Zariski-singular support} of $a$, see \cite{AW3}.

For proofs of the statements in this paragraph we refer to \cite[Section~3]{AW3}, see also \cite[Section~2]{AS}.
Let $a$ be an almost semi-meromorphic current on $X$ and let $\mu\in\PM(X)$. Then there is a 
unique pseudomeromorphic current $T$ on $X$ coinciding with $a\wedge\mu$ outside of $ZSS(a)$ and such that
$\ett_{ZSS(a)}T=0$. If $h$ is a holomorphic tuple, or section of a 
Hermitian vector bundle, such that $\{h=0\}=ZSS(a)$, then $T=\lim_{\epsilon\to 0}\chi(|h|^2/\epsilon)a\wedge\mu$;
henceforth we will write $a\wedge\mu$ in place of $T$. One defines $\debar a\wedge \mu$ so that Leibniz' rule holds,
i.e., $\debar a\wedge \mu:=\debar(a\wedge\mu)-(-1)^{\textrm{deg}\, a}a\wedge\debar\mu$.
If $\mu\in\W(X)$ then $a\wedge\mu\in \W(X)$; in this case 
$a\wedge\mu=\lim_{\epsilon\to 0}\chi(|h|^2/\epsilon)a\wedge\mu$	if $h$ is any generically non-vanishing 
holomorphic section of a Hermitian vector bundle such that $\{h=0\}\supset ZSS(a)$.
If $\mu$ is almost semi-meromorphic then $a\wedge \mu$ is almost semi-meromorphic and, in fact,
$a\wedge\mu=(-1)^{\textrm{deg}\, a \, \textrm{deg}\, \mu}\mu\wedge a$.

Let $X$ be an analytic subset of pure codimension $\kappa$ of some complex $N$-dimensional manifold $D$.
The subsheaves of $\PM_D$ of germs of $\debar$-closed $(k,\kappa)$-currents, $k=0,\ldots,N$, with support on 
$X$ are the sheaves of Coleff-Herrera currents with support on $X$ and are denoted $\CH_X^k$. Coleff-Herrera currents were
originally introduced by Bj\"{o}rk as the $\debar$-closed currents $\mu$ on $D$ of bidegree $(N,\kappa)$
such that $\bar{h}\mu=0$ for any holomorphic function $h$ vanishing on $X$ and with the SEP with respect to
$X$, see, e.g., \cite{JEBabel}. It is proved in \cite{Auniqueness} that the definitions are equivalent.
The model example is the Coleff-Herrera product: Assume that $f_1,\ldots,f_{\kappa}\in\hol(D)$ defines a regular
sequence. Then the iteratively defined product 
$\debar(1/f_1)\wedge\cdots\wedge\debar(1/f_{\kappa})$ is the Coleff-Herrera product originally introduced
by Coleff and Herrera in \cite{CH} in a slightly different way; cf.\ also \cite{JEBHS}.

Let us also notice that if $X$ and $Z$ are reduced pure-dimensional complex spaces and
$\mu\in\PM(X)$, then $\mu\otimes 1\in \PM(X\times Z)$, see, e.g., \cite[Section~2]{AS}. We will usually 
omit ``$\otimes 1$'' and simply write, e.g., $\mu(\zeta)$ to denote which coordinates
$\mu$ depends on.

\subsection{Residue currents associated with generically exact complexes}\label{AWcurrsektion}
Let $E_j$, $j=0,\ldots,M$, be trivial vector bundles over an open subset of $\C^N$, 
let $f_j\colon E_j\to E_{j-1}$ be holomorphic mappings, and assume that 
\begin{equation}\label{eq:resol2}
0\to E_M \stackrel{f_M}{\longrightarrow} \cdots \stackrel{f_2}{\longrightarrow} E_1
\stackrel{f_1}{\longrightarrow} E_0 \stackrel{f_0}{\longrightarrow} 0,
\end{equation}
is a complex that is pointwise exact outside of an analytic subset $V$ of positive codimension. 
The bundle $E:=\oplus_jE_j$ gets a natural superstructure by setting $E^+:=\oplus_jE_{2j}$
and $E^-:=\oplus_jE_{2j+1}$.
Following \cite{AW1} we define currents $U$ and $R$ with values in $\textrm{End}(E)$ associated with 
\eqref{eq:resol} and a choice of Hermitian metrics on the $E_k$.\footnote{That a current takes values
in a vector bundle $F$ means that it acts on test-forms with values in $F^*$.} Notice that $\textrm{End}(E)$
gets an induced superstructure and so spaces of forms and currents with values in $E$ or $\textrm{End}(E)$
get superstructures as well.  
Let $f:=\oplus_jf_j$ and set 
$\nabla:=f-\debar$, which then becomes an odd mapping on spaces of forms or currents with values in $E$
such that $\nabla^2=0$;
notice that $\nabla$ induces an odd mapping $\nabla_{\textrm{End}}$ on $\textrm{End}(E)$-valued forms or currents
such that $\nabla^2_{\textrm{End}}=0$.
Outside of $V$, let $\sigma_k\colon E_{k-1}\to E_k$ be the pointwise minimal inverse of $f_k$, i.e.,
for each $z\notin V$, 
\begin{equation*}
\sigma_k(z) f_k(z)=\Pi_{(\textrm{Ker} f_k(z))^{\perp}}, \quad
f_k(z)\sigma_k(z) = \Pi_{\textrm{Im} f_k(z)},
\end{equation*}
where $\Pi$ denotes orthogonal projection.
Set $\sigma:=\sigma_1+\sigma_2+\cdots$ and let $u:=\sigma+\sigma\debar\sigma + \sigma(\debar\sigma)^2+\cdots$.
Notice that $u=\sum_{0<\ell}\sum_{0\leq k<\ell}u^k_{\ell}$, where 
$u^k_{\ell}:=\sigma_{\ell}\debar\sigma_{\ell-1}\cdots\debar\sigma_{k+1}$ is a smooth 
$\textrm{Hom}(E_k,E_{\ell})$-valued $(0,\ell-k-1)$-form outside of $V$.
One can show that $\nabla_{\textrm{End}}u=\textrm{Id}_{E}$.
We extend $u$ as a current across $V$ by setting
$U:=\lim_{\epsilon\to 0}\chi(|F|^2/\epsilon)u$,
where $F$ is a (non-trivial) holomorphic tuple vanishing on $V$, cf., \cite[Section~2]{AW1} and \cite[Theorem~5.1]{ABullSci}. 
As with $u$ we will write $U=\sum_{0<\ell}\sum_{0\leq k<\ell}U^k_{\ell}$, where now $U^k_{\ell}$ is a 
$\textrm{Hom}(E_k,E_{\ell})$-valued $(0,\ell-k-1)$-current.
\begin{remark}\label{ASMremark}
The procedure of taking pointwise minimal inverses produce almost semi-meromorphic currents,
see, e.g., \cite[Section~4]{AW3}. Thus
the $\sigma_j$ have almost semi-meromorphic extensions across $V$ and, letting 
$\sigma_j$ denote the extension as well, we have 
$U^k_{\ell}:=\sigma_{\ell}\debar\sigma_{\ell-1}\cdots\debar\sigma_{k+1}$, where 
the products are in the sense of Section~\ref{PMsektion} above.
In particular, each $U^k_{\ell}$ is an almost semi-meromorphic current in (some domain in) $\C^N$.
\end{remark}

The current $R$ is defined by the equation $\nabla_{\textrm{End}} U=\textrm{Id}_{E} - R$,
and hence, $R$ is supported on $V$ and $\nabla_{\textrm{End}}R=0$. 
Notice that $R$ is an almost semi-meromorphic current plus $\debar$ of such a current.
One can check that 
\begin{equation}\label{Rreg}
R=\lim_{\epsilon\to 0}\, \big(1-\chi(|F|^2/\epsilon)\big)\textrm{Id}_{E} + \debar\chi(|F|^2/\epsilon)\wedge u.
\end{equation}
We write $R=\sum_{0<\ell}\sum_{0\leq k<\ell}R^k_{\ell}$, where $R^k_{\ell}$ is a 
$\textrm{Hom}(E_k,E_{\ell})$-valued $(0,\ell-k)$-current. 

% ???? If $\tilde{U}$ and $\tilde{R}$ are currents such that ............ and $\tilde{R}\varphi=0$ then
% $\varphi$ is in ......???????

\smallskip

Now consider the sheaf complex 
\begin{equation}\label{eq:resol}
0\to \hol(E_M) \stackrel{f_M}{\longrightarrow} \cdots \stackrel{f_2}{\longrightarrow} \hol(E_1)
\stackrel{f_1}{\longrightarrow} \hol(E_0)
\end{equation}
associated with \eqref{eq:resol2}. Assume that \eqref{eq:resol} is exact so that it provides a free 
resolution of the sheaf $\F=\hol(E_0)/\Image f_1$. Recall that any coherent sheaf is of this form 
and has a free resolution locally. By definition, $\F$ has (co)dimension $r$ if 
the associated primes of each stalk $\F_x$ all have (co)dimension $\leq r$ ($\geq r$);
$\F$ has pure (co)dimension if all associated primes are equidimensional.
Let $Z_k$ be the set where $f_k$ does not have optimal rank; it is well-known that the $Z_k$ 
are analytic and independent of the choice of free resolution, thus invariants of $\F$.
Let $\kappa=\text{codim}\,\F$. 
By, e.g., \cite[Corollary~20.12]{Eis}, 
\begin{equation*}
\cdots \subset Z_k \subset Z_{k-1}\subset \cdots \subset Z_{\kappa+1}\subsetneq Z_{\kappa}=\cdots=Z_1.
\end{equation*}
Moreover, by \cite[Corollary~20.14]{Eis}, 
$\textrm{codim}\, Z_k\geq k+1$ for $k\geq \kappa+1$ if and only if $\F$ has pure codimension $\kappa$.
We recall also that $\F$ is Cohen-Macaulay if and only if 
$Z_k=\emptyset$ for $k\geq \kappa+1$, i.e., if and only if there is a resolution \eqref{eq:resol} of $\F$
with $M=\kappa$.

By definition, see \cite[\S 1]{SiuTraut}, the singularity subvarieties $S_{\ell}(\mathscr{G})$ of 
$\mathscr{G}:=\F\restriction_{Z_1}$ are the set of points $x\in Z_1$ such that
$\text{depth}_{\hol_{Z_1,x}}(\mathscr{G}_x)\leq \ell$.
It is straightforward to check that $Z_k$ is the set of points
$x\in \C^N$ such that the projective dimension of $\F_x$ is $\geq k$ and so,
from the Auslander-Buchsbaum formula, it follows that $S_{N-\ell}(\mathscr{G})=Z_{\ell}$.

It is proved in \cite{AW1} that if \eqref{eq:resol} is exact and $R$ is the associated current,
then $R=\sum_{\ell\geq \kappa}R^0_{\ell}$.
Moreover, a section $\varphi$ of $\hol(E_0)$ is in $\Image f_1$ if and only if (the $E$-valued) current $R\varphi$
vanishes. In what follows we will only be concerned with currents associated to exact complexes \eqref{eq:resol}.
We will therefore write $R_{\ell}:=R^0_{\ell}$.

\begin{example}\label{koszulex}
The model example is the Koszul complex: Let $f_1,\ldots,f_{\kappa}\in\hol(D)$ ($D$ a domain in $\C^N$) be
a regular sequence and let \eqref{eq:resol} with $M= \kappa$ be the associated Koszul complex,
which then is a resolution of $\hol/\langle f_1,\ldots,f_{\kappa}\rangle$.
With the trivial metric on the bundles $E_j$ the resulting $R$ is $R_{BM}\wedge e_{\kappa}\wedge e_0^*$,
where $R_{BM}$ is the residue current of Bochner-Martinelli type introduced in \cite{PTY} and 
$e_0$ and $e_{\kappa}$ are suitable frames for the line bundles $E_0$ and $E_{\kappa}$ respectively.
It is shown in \cite{PTY}, see also \cite{Auniqueness}, that $R_{BM}$ equals the Coleff-Herrera product in the present situation.
By \cite[Theorem~4.1]{AW1}, $R$ is in fact independent of the 
choice of Hermitian metric and so the above procedure always produce the Coleff-Herrera product (times
$e_{\kappa}\wedge e_0^*$) in the case of regular sequences.
\end{example}

%%%%%%%%%%%%%%%%%%%%%%%%%%%%%%%%%%%%%%%%%%%%%%%%%%%%%%%%%%%%%%%%%%%%%%%%%%%%%%%%%%%%%%%%%%%%%%%%%%%%%%%%%%%%%%%%%%%%%%%%%%%%%%%%%%%%%%%%

\section{Strongly holomorphic $p$-forms on $X$}\label{strong}
Let $X=\{h_1=\cdots=h_{r}=0\}$ be a pure $n$-dimensional analytic subset of a neighborhood of $0$ in $\C^N$ and set
$\kappa:=N-n$; assume that $0\in X$. Let $\{\varphi_j\}$ and $\{\psi_j\}$ be finite sets of generators for 
$\mathit{\Omega}^p$ and $\mathit{\Omega}^{p-1}$ respectively and let $\tilde{\J}^p_X\subset \mathit{\Omega}^p$ be the coherent subsheaf
generated over $\hol$ by $\{h_i\varphi_j\}$ and $\{dh_i\wedge \psi_j\}$. It is clear that 
$\tilde{\J}^p_{X,x}=\mathit{\Omega}^p_x$ for $x$ outside of $X$, that $\textrm{codim}\,\mathit{\Omega}^p/\tilde{\J}^p_{X}=\kappa$, and
that $\mathit{\Omega}^p/\tilde{\J}^p_{X}$ is the standard sheaf of holomorphic $p$-forms on $X_{reg}$. By definition, 
$\mathit{\Omega}^p/\tilde{\J}^p_X\restriction_X=:\mathit{\Omega}_X^p$
is the sheaf of germs of K\"{a}hler-Grothendieck differential $p$-forms on $X$.
In general, $\mathit{\Omega}^p/\tilde{\J}^p_X$ has torsion.

\begin{example}
If $X=\{z_1^2=z_2^3\}\subset \C^3$ then $\varphi=2z_2dz_1 - 3z_1dz_2$ is a torsion element;
it is not in $\tilde{\J}^1_{X,0}$ even though the pullback of $\varphi$ to $X_{reg}$ vanishes.
More generally, if $X$ is a germ of an arbitrary reduced planar singular curve at $0\in\C^2$,
then, as one can check, $\mathit{\Omega}^1_{0}/\tilde{\J}^1_{X,0}$ always has torsion.
% Let $C=\{f=0\}$ be a curve in a neighborhood of $0$ in $\C^2$ such that $0\in C_{sing}$. Then ....... is a free resolution
% of $\Om^1_{x}/\tilde{\J}^1_{C}$; in fact, the codimensions of the BEF-sets descend in the right way, see
% [Eisenbud]. However, $\Om^2_{x}/\tilde{\J}^p_{C}$ is not pure dimensional since if it were, then 
% the second BEF-set would be empty by Cor 20.14 in [Eisenbud]. But it contains $0$!
\end{example}

From a primary decomposition of $\tilde{\J}^p_{X,0}$ we see that there are coherent sheaves $\J^p_X$ and 
$\mathscr{S}^p_X$ in a neighborhood $\mathcal{U}$ of $0$ such that
$\tilde{\J}^p_X = \J^p_X \cap \mathscr{S}^p_X$,
$\mathit{\Omega}^p/\J^p_X$ has pure codimension $\kappa$, and $\mathit{\Omega}^p/\mathscr{S}^p_X$ has codimension 
$> \kappa$.\footnote{The reader familiar with gap-sheaves will recognize $\mathcal{J}^p_X$ as the 
relative gap-sheaf of $\tilde{\mathcal{J}}^p_X$ in $\Om^p$ with respect to $X_{sing}$;
see, e.g., \cite[p. 47]{SiuTraut}.}
It follows that $\tilde{\J}^p_X = \J^p_X$ outside of an analytic set of codimension $>\kappa$;
outside of an analytic set of codimension $>1$ in $X$ thus
$\tilde{\J}^p_X = \J^p_X$. Hence, the pullback of any section of $\J^p_X$ to $X_{reg}$ vanishes.
On the other hand, since $\mathit{\Omega}^p/\J^p_X$ has pure dimension it follows that any section $\varphi$ of $\mathit{\Omega}^p$ such that
the pullback of $\varphi$ to $X_{reg}$ vanishes in fact is a section of $\J^p_X$; this is well-known and also follows from
Proposition~\ref{fundprop} below. Thus, $\J^p_X$ is the sheaf of germs of holomorphic $p$-forms $\varphi$ in $\mathcal{U}$
such that $\varphi\wedge [X]=0$. We set $\Om_X^p:=\mathit{\Omega}^p/\J^p_X\restriction_X$ and we call the sections of 
$\Om_X^p$ \emph{strongly holomorphic $p$-forms}; these are thus precisely the $p$-forms on $X_{reg}$ that locally are the pullback 
to $X_{reg}$ of holomorphic $p$-forms in ambient space. Since $\J_X^p$ is coherent by construction, $\Om_X^p$ is coherent 
and it is readily checked that $\Om_X^p$ equals
$\mathit{\Omega}^p_X$ modulo torsion. Notice that the sections of $\Om_X^p$ define $\debar$-closed currents on $X$
with the SEP. We remark the the strongly holomorphic forms have been studied by several authors, e.g., in \cite{Ferrari}
and \cite{HP}.

For simplicity we will for the rest of this section assume that $X$ and $\J^p_X$ are defined in a neighborhood of 
the closure of the unit ball $\B$ of $\C^N$ and we denote the inclusion $X\hookrightarrow \B$ by $i$.
Moreover, we let \eqref{eq:resol} be a resolution of $i_*\Om^p_X=\mathit{\Omega}^p/\J^p_X$ with 
$E_0=\Lambda^{p,0}T^*\C^N$ so that $\hol(E_0)=\mathit{\Omega}^p$; recall also the associated sets
$Z_k$, cf.\ Section~\ref{AWcurrsektion}. 
Since $\Om^p_X$ has pure codimension we have $\textrm{codim}\, Z_k \geq k +1$, for $k=\kappa+1, \kappa+2,\ldots$,
and in particular $Z_N=\emptyset$. Hence,  we can, 
and will, assume that $M\leq N-1$ in \eqref{eq:resol}. 
The resolution \eqref{eq:resol} induces a complex \eqref{eq:resol2} that
is pointwise exact outside of $X$. A choice of Hermitian metrics on the $E_j$ gives us associated 
$\textrm{Hom}(E_0,E)$-valued currents
$U$ and $R$ so that, in particular, a holomorphic $p$-form $\varphi$ is a section of $\J^p_X$ if and only if 
the $E$-valued current $R\varphi$ vanishes.

\begin{example}\label{ex:smooth}
Assume that $X=\{w_1=\cdots=w_{\kappa}=0\}$, where $(z_1,\ldots,z_n; w_1,\ldots,w_{\kappa})$ are 
local coordinates in an open subset $U$ of $\C^N$. A basis for the $(p,0)$-forms in $U$ is given by
the union of $\{dz_I\wedge dw_J\}$, where $I$ and $J$ range over increasing multiindices
such that $|I|+|J|=p$. Let $E_0'$ and $E_0''$ be the subbundles of $\Lambda^{p,0}T^*U$ generated by
$dz_I$, $|I|=p$, and $dz_J\wedge dw_K$, $|J|<p$, respectively. It is clear that $\J^p_X$
is generated by $w_idz_J$, $i=1,\ldots,\kappa$, $|J|=p$ and $dz_I\wedge dw_J$, $|J|\geq 1$.
To get a resolution of $\Om^p_X$ we let, for each increasing multiindex $J\subset \{1,\ldots,n\}$ with $|J|=p$,
$\big(E^J_{\bullet}, f^J_{\bullet}\big)$ be the Koszul complex corresponding 
to $w_1,\ldots,w_{\kappa}$ and we identify $E^J_0$ with the line bundle generated by $dz_J$;
notice that $\oplus_{|J|=p}E^J_0=E_0'$.
It is well-known that $\big(\hol(E^J_{\bullet}), f^J_{\bullet}\big)$ is a resolution 
of the quotient $\hol dz_J/(w_1,\ldots,w_{\kappa})\hol dz_J$.
Let $\big(E_{\bullet}',f'_{\bullet}\big)$ be the direct sum of the complexes $\big(E^J_{\bullet}, f^J_{\bullet}\big)$
over all increasing multiindices $J$ with $|J|=p$.
Then
\begin{equation}\label{sumKoszul}
0\to \hol(E_{\kappa}') \stackrel{f_{\kappa}'}{\longrightarrow} \cdots \stackrel{f_{3}'}{\longrightarrow} \hol(E_2')
\stackrel{f_{2}'}{\longrightarrow} \hol(E_1')\oplus \hol(E_0'') 
\stackrel{f_{1}' \oplus \textrm{Id}}{\longrightarrow} \hol(E_0')\oplus \hol(E_0'')
\end{equation}
is a resolution of $\Om^p_X$ since \eqref{sumKoszul} is exact (as a direct sum of exact complexes) and the 
cokernel of the map $f_{1}' \oplus \textrm{Id}$ equals $\Om^p_X$.

Since $w_1,\ldots,w_{\kappa}$ is a regular sequence it follows that, for any choice of 
Hermitian metrics on the $E_i^J$, the current $R^J$ associated with $\big(E^J_{\bullet}, f^J_{\bullet}\big)$
equals 
\begin{equation*}
R^J = \varepsilon^J \otimes (dz_J)^*\otimes \debar \frac{1}{w_1}\wedge \cdots \wedge \debar \frac{1}{w_{\kappa}}, 
\end{equation*}
where $\varepsilon^J$ is a frame for $E^J_{\kappa}$, $(dz_J)^*$ is the dual of $dz_I$, and 
$\debar (1/w_1)\wedge \cdots \wedge \debar (1/w_{\kappa})$ is the Coleff-Herrera product,
cf.\ Example~\ref{koszulex}. 
Choosing a metric that respects the direct sum structure we get that
the current $R$ associated with \eqref{sumKoszul} equals
\begin{equation*}
R = \sum_{|J|=p}' \varepsilon^J \otimes (dz_J)^*\otimes \debar \frac{1}{w_1}\wedge \cdots \wedge \debar \frac{1}{w_{\kappa}}.
\end{equation*} 
Set $dz=dz_1\wedge\cdots\wedge dz_n$ and $dw=dw_1\wedge\cdots\wedge dw_{\kappa}$ and notice that by the Poincar\'e-Lelong formula
\begin{equation*}
R\wedge dw\wedge dz = \pm (2\pi i)^{\kappa} \sum_{|J|=p}' \varepsilon^J \otimes (dz_J)^*\otimes [X]\wedge dz.
\end{equation*}
If $\varphi$ is a $(p,n)$-test form in $U$ then we can view $\varphi$ as an $E_0$-valued $(0,n)$-test
form, where $E_0=E_0'\oplus E_0''=\Lambda^{p,0}T^*U$. With this point of view we see that 
\begin{equation*}
R\wedge dw\wedge dz . \varphi = \pm (2\pi i)^{\kappa} \sum_{|J|=p}' \varepsilon^J \int_X dz_{J^c}\wedge \varphi,
\end{equation*}
where $J^c=\{1,\ldots,n\}\setminus J$.
\end{example}

The preceding example indicates that the $(N,*)$-current $R\wedge dz$ is of the form $i_*\mu$ for some 
$(n-p,*)$-current $\mu$ on $i\colon X \hookrightarrow \B$; here and in the rest of the paper, $dz:=dz_1\wedge\cdots\wedge dz_N$. 
At first sight this seems to contradict 
the first paragraph of Section~\ref{prelim}. To shed some light on this notice first that $R_k\wedge dz$ is 
a (distribution-valued) section of $E_k\otimes E^*_0\otimes \Lambda^{N,k}T^*\B$ since
$\text{Hom}(E_0,E_k)\simeq E_k\otimes E^*_0$, and recall that 
$E_0=\Lambda^{p,0}T^*\B$. Interior multiplication induces a natural
isomorphism $E^*_0\otimes \Lambda^{N,k}T^*\B \to \Lambda^{N-p,k}T^*\B$ and moreover, if $\varphi$ is an
$E_0$-valued $(0,N-k)$-form then we can also view it as a $(p,N-k)$-form $\tilde{\varphi}$. We get a
diagram
\begin{equation*}
\begin{array}{ccc}
E^*_0\otimes \Lambda^{N,k}T^*\B & \stackrel{\varphi}{\longrightarrow} & \Lambda^{N,N}T^*\B \\
\downarrow & & || \\
\Lambda^{N-p,k}T^*\B & \stackrel{\wedge\tilde{\varphi}}{\longrightarrow} & \Lambda^{N,N}T^*\B
\end{array}
\end{equation*}
where $\varphi$ also denotes the natural map induced by $\varphi$, and the map $\wedge\tilde{\varphi}$ 
is defined by taking wedge product with $\tilde{\varphi}$. The diagram commutes, as can be checked, and
therefore we can view $R\wedge dz$ either as an $(N,*)$-current
with values in $\textrm{Hom}(E_0,E)\simeq E\otimes E_0^*$ or as an $(N-p,*)$-current
with values in $E$; with the first viewpoint $R\wedge dz$ acts naturally on $E_0\otimes E^*$-valued 
$(0,*)$-test forms and with the second one it acts on $E^*$-valued $(p,*)$-test forms
and the result is the same. For future reference we also note that with the first point of view $R\wedge dz$
can be naturally multiplied with smooth $E_0$-valued $(0,*)$-forms yielding $E$-valued currents; with the second point of view
$R\wedge dz$ can be naturally multiplied with scalar-valued $(p,*)$-forms yielding the same $E$-valued currents.
Unless explicitly said, we will use the second point of view (even though the notation might suggest otherwise). 

The following proposition is the analogue of \cite[Proposition~3.3]{AS} and the proof is essentially the same,
we therefore omit it.

%%%%%%%%%%%%
%FUNDAMENTAL PROP
%%%%%%%%%%%%

\begin{proposition}\label{fundprop}
There is a unique almost semi-meromorphic current 
$\omega=\omega_0+\omega_1+\cdots+\omega_{n-1}$ on $X$, where $\omega_k$ is an $E_{\kappa+k}$-valued $(n-p,k)$-current,
such that 
\begin{equation*}
R\wedge dz = i_* \omega.
\end{equation*}
The current $\omega$ has the following additional structure.
\begin{itemize}
\item[(i)] If $\Om^p_X$ is Cohen-Macaulay,
then $\omega_0$ is an $E_{\kappa}$-valued section of $\omega^{n-p}_X$ over $X$. In general, 
there is a tuple $\tilde{\omega}_0$ of sections of $\omega^{n-p}_X$ over $X$
and a tuple $\alpha_0$ of almost semi-meromorphic $E_{\kappa}$-valued $(0,0)$-current in $\B$, smooth 
outside of $Z_{\kappa+1}$, such that
$\omega_0=\alpha_{0\restriction_X} \cdot \tilde{\omega}_0$ as currents on $X$. 

\item[(ii)] For $k\geq 1$ there are almost semi-meromorphic $(0,1)$-currents 
$\alpha_k$ in $\B$ with values in $\textrm{Hom}(E_{\kappa+k-1},E_{\kappa+k})$ that are smooth outside of $Z_{\kappa+k}$
and such that $\omega_{k}=\alpha_{k\restriction_X} \omega_{k-1}$ as currents.
\end{itemize}
\end{proposition}

The form $\omega$ will be called an $n-p$-\emph{structure form}.

Since $R\wedge dz=i_*\omega$, where $\omega$ is almost semi-meromorphic on $X$, it follows that $R$ has the SEP
with respect to $X$. In particular, if $\varphi$ is a holomorphic $p$-form in ambient space such that 
the pullback of $\varphi$ to $X_{reg}$ vanishes, then the ($E$-valued) current $R\varphi$ vanishes, i.e.,
$\varphi$ is a section of $\J^p_X$.

\begin{lemma}\label{divlemma}
If $\varphi$ is a smooth $(n-p,q)$-form on $X$ then there is a smooth $(0,q)$-form $\phi$ 
on $X$ with values in $E_{\kappa}^*\restriction_X$ such that $\varphi=\omega_0\wedge \phi$.
\end{lemma}
\begin{proof}
Consider a smooth extension of $\varphi$ to $\B$; it can be written in the form $\sum_j\varphi_j'\wedge\varphi_j''$
where $\varphi_j'$ is a holomorphic $n-p$-form in $\B$ and $\varphi_j''$ is a smooth $(0,q)$-form in $\B$. 
The $(N-p,\kappa)$-current $\varphi_j'\wedge [X]$ can be viewed as a section of $\Hom_{\hol}(\Om_X^p,\CH^N_X)$
via $\Om_X^p\ni\psi\mapsto \psi\wedge \varphi_j'\wedge [X]$. By Proposition~\ref{BHPprop} below (with $p$ and $n-p$ interchanged)
there is a section $\xi_j$
of $\hol(E_{\kappa}^*)$ such that $i_*(i^*\xi_j\cdot \omega_0)=\varphi_j'\wedge [X]$. 
It follows that 
$\varphi= \sum_ji^*\xi_j \cdot \omega_0\wedge i^*\varphi_j''$.
\end{proof}

%%%%%%%%%%%%%%%%%%%%%%%%%%%%%%%%%%%%%%%%%%%%%%%%%%%%%%%%%%%%%%%%%%%%%%%%%%%%%%%%%%%%%%%%%%%%%%%%%%%%%%%%%%%%%%%%%%%%%%%%%%%%%%%%%%%%%%%%%%

\section{Barlet's sheaf $\omega_X^p$.}\label{BHPsheaf}
The sheaf $\omega^p_X$ was introduced by Barlet in \cite{Barlet} as the kernel of 
a natural map $j_*j^*\mathit{\Omega}_X^p \to \HH^1_{X_{sing}}\big(\Ext_{\hol}^{\kappa}(\hol_X, \mathit{\Omega}^{\kappa+p})\big)$,
where $j\colon X_{reg}\hookrightarrow X$ is the inclusion. It is proved, \cite[Proposition~4]{Barlet}, 
that the sections of $\omega_X^p$ can be identified with
the holomorphic $p$-forms on $X_{reg}$ that have an extension to $X$ as a $\debar$-closed current with the SEP. 
Moreover, it is shown that $\omega_X^p$ is coherent and so 
$\omega_X^p/\Om_X^p$ is a coherent sheaf supported on $X_{sing}$. Hence, locally, for a suitable generically non-vanishing
holomorphic function $h$, one has $h\omega_X^p\subset \Om_X^p$. It follows that $\omega_X^p$ can be identified with
the sheaf of germs of meromorphic $p$-forms on $X$ that are $\debar$-closed considered as principal value currents;
we will use this as the definition of $\omega_X^p$.
This analytic point of view was emphasized and explored by Henkin and Passare, \cite{HP},
and therefore we often call sections of $\omega_X^p$ Barlet-Henkin-Passare holomorphic $p$-forms. 

From Barlet's definition, since $j_*j^*\mathit{\Omega}_X^p$ is torsion free, (and from the one we use as well) 
it is clear that $\omega_X^p$ is torsion free. 
Moreover, from \cite[p.\ 195]{Barlet} it follows that if $\textrm{codim}_X\, X_{sing} \geq 2$, 
then any holomorphic $p$-form on $X_{reg}$
extends (necessarily uniquely) to a section of $\omega_X^p$ over $X$. Thus, by \cite[Proposition~1.6]{Hart}, if $X$ is normal then 
$\omega_X^p$ is reflexive. On a normal space the reflexive hull of any reasonable sheaf of holomorphic $p$-forms therefore
coincides with $\omega_X^p$.

Let $X$ be a pure $n$-dimensional analytic subset of a neighborhood of $\overline{\B}\subset \C^N$, $\kappa=N-n$, and
let \eqref{eq:resol} be a resolution of $i_*\Om^{n-p}_X=\mathit{\Omega}^{n-p}/\J^{n-p}_X$ in $\B$; notice that now
$\hol(E_0) =\mathit{\Omega}^{n-p}$. Let $R=R_{\kappa}+\cdots$ be the current
associated with \eqref{eq:resol} (for some choice of Hermitian metrics),
let $i_*\omega= R\wedge dz$, and recall that  
$\omega_0$ is a $(p,0)$-current on $X$ with values in $E_{\kappa\restriction_X}$; 
cf., Proposition~\ref{fundprop} and the paragraph preceding it.
By dualizing and tensoring by $\mathit{\Omega}^N$ we get the complex 
$\big(\hol(E^*_{\bullet})\otimes \mathit{\Omega}^N, f_{\bullet}^*\otimes \textrm{Id}\big)$
with associated cohomology sheaves 
$\mathscr{H}^{\ell}\big(\hol(E^*_{\bullet})\otimes \mathit{\Omega}^N\big) \simeq \Ext^{\ell}_{\hol}\big(\Om^{n-p}_X, \mathit{\Omega}^N\big)$.
Let $\xi\in \hol(E^*_{\kappa})$ be such that $f^*_{\kappa+1}\xi=0$. Then
\begin{equation*}
\debar (\xi\cdot i_*\omega_0) = \xi\cdot\debar R_{\kappa} \wedge dz = \xi\cdot f_{\kappa+1}R_{\kappa+1} \wedge dz
=f_{\kappa+1}^*\xi \cdot R_{\kappa+1}\wedge dz=0,
\end{equation*}
and it follows that the current $i^*\xi \cdot \omega_0$ is $\debar$-closed on $X$. Hence, 
$i^*\xi\cdot \omega_0$ is a section of $\omega^p_X$. If $\xi=f^*_{\kappa} \xi'$ one checks in a similar
way that $i^*\xi \cdot \omega_0=0$ and we see that we have a mapping 
\begin{equation}\label{BHPmap}
\mathscr{H}^{\kappa}\big(\hol(E^*_{\bullet})\otimes \mathit{\Omega}^N\big)  \rightarrow  \omega^p_X, \quad
\left[\xi \right] \otimes dz  \mapsto   i^*\xi \cdot \omega_0.
\end{equation}

\begin{proposition}\label{BHPprop}
The mapping \eqref{BHPmap} is an isomorphism and it induces a natural isomorphism
$\displaystyle{\omega_X^p\simeq\Ext^{\kappa}_{\hol}\big(\Om^{n-p}_X, \mathit{\Omega}^N\big)}$.
\end{proposition}

\begin{proof}
Let $\varphi$ be a section of $\omega_X^p$. Then $i_*\varphi$ is a $\debar$-closed 
$(\kappa+p,\kappa)$-current in $\B$ and it induces a map $\mathit{\Omega}^{n-p}\to \CH^N_X$ by
\begin{equation}\label{oslo}
\psi\mapsto i_*\varphi\wedge \psi,
\end{equation}
whose kernel clearly contains $\J^{n-p}_X$. Hence, \eqref{oslo}
induces a map $\mathit{\Omega}^{n-p}/\J^{n-p}_X\to \CH^N_X$. Thus, we get a map
$\omega_X^p\to \Hom_{\hol}(\Om^{n-p}_X, \CH^N_X)$, which one easily checks is injective.
In view of \eqref{BHPmap} we get a commutative diagram
\begin{equation}\label{diagramm}
\xymatrix{
\mathscr{H}^{\kappa}\big(\hol(E^*_{\bullet})\otimes \mathit{\Omega}^N\big) \ar[r] \ar[dr] & \omega^p_X \ar[d] \\
& \Hom_{\hol}\big(\Om^{n-p}_X, \CH^N_X\big),
}
\end{equation}
where the diagonal map is the composition, i.e., the map given by 
$\left[\xi\right]\otimes dz \mapsto \xi\cdot R_{\kappa}\wedge dz$, where we here view $R_{\kappa}\wedge dz$
as a $\textrm{Hom}(E_0,E_{\kappa})$-valued $(N,\kappa)$-current.
By \cite[Theorem~1.5]{ANoetherDual} this map is an isomorphism and
since the vertical map is injective it follows that both the horizontal map and the vertical map are isomorphisms.
From \emph{ibid}.\ we also know that the diagonal map is independent of the choices of Hermitian resolution
of $\Om^{n-p}_X$ and of $dz$. 
\end{proof}

Barlet has recently found an elegant algebraic proof of the isomorphism
$\omega^p_X\simeq\Ext^{\kappa}_{\hol}\big(\Om^{n-p}_X, \mathit{\Omega}^N\big)$ of Proposition~\ref{BHPprop}
that he has communicated to us
and generously let us include here.

\begin{proof}[Alternative proof of Proposition~\ref{BHPprop}.]
There is a natural map $\mathit{\Omega}^{n-p}_X\to\Om_X^{n-p}$; denote the kernel by $\mathscr{T}$ and notice that it has 
codimension $>\kappa$. Thus, $\Ext^{k}_{\hol}(\mathscr{T},\mathit{\Omega^N})=0$
for $k\leq \kappa$. Applying the functor $\Hom_{\hol}(-,\mathit{\Omega}^{N})$ to the exact sequence
$0\to\mathscr{T}\to \mathit{\Omega}^{n-p}_X\to\Om_X^{n-p}\to 0$
we get a long exact sequence of $\Ext$-sheaves. From this, and the vanishing of $\Ext^{k}(\mathscr{T},\mathit{\Omega^N})$
for $k\leq \kappa$, it follows that
$\Ext^{\kappa}_{\hol}\big(\Om^{n-p}_X, \mathit{\Omega}^N\big)\simeq\Ext^{\kappa}_{\hol}\big(\mathit{\Omega}^{n-p}_X, \mathit{\Omega}^N\big)$. 

Let $\G:=(d\J_X^0\wedge\mathit{\Omega}^{n-p-1})\cap(\J_X^0\mathit{\Omega}^{n-p})$, let
$\F:=d\J_X^0\wedge\mathit{\Omega}^{n-p-1}/\G$, and notice that $\F$ and $\G$ are $\hol_X$-modules;
$\mathcal{J}_X^0\subset\hol$ is the ideal defining $X$, cf.\ Section~\ref{strong}.
We have a natural short exact sequence of $\hol_X$-modules in $\B$
\begin{equation*}
0\to \F \longrightarrow \hol_X \otimes \mathit{\Omega}^{n-p} \longrightarrow \mathit{\Omega}_X^{n-p} \to 0.
\end{equation*}
Applying $\Hom_{\hol}(-,\mathit{\Omega}^{N})$ we again obtain a long exact sequence
of $\Ext$-sheaves. Since $\text{codim}\, X=\kappa$ these sheaves vanish until level $\kappa$ and in particular one 
gets the exact sequence
\begin{equation*}
0\to \Ext^{\kappa}_{\hol}(\mathit{\Omega}_X^{n-p},\mathit{\Omega}^{N}) \longrightarrow
\Ext^{\kappa}_{\hol}(\hol_X\otimes \mathit{\Omega}^{n-p}, \mathit{\Omega}^{N}) \stackrel{b}{\longrightarrow}
\Ext^{\kappa}_{\hol}(\F,\mathit{\Omega}^{N}).
\end{equation*}
Since $\mathit{\Omega}^{n-p}$ is a free $\hol$-module and since $\Ext^{\kappa}_{\hol}(\hol_X,\mathit{\Omega}^{N})\simeq i_*\omega_X^n$ 
by \cite[Lemma~4]{Barlet}, one has
\begin{equation*}
\Ext^{\kappa}_{\hol}(\hol_X\otimes \mathit{\Omega}^{n-p}, \mathit{\Omega}^{N}) \simeq
\Hom_{\hol}(\mathit{\Omega}^{n-p}, \Ext^{\kappa}_{\hol}(\hol_X,\mathit{\Omega}^{N})) \simeq
\Hom_{\hol}(\mathit{\Omega}^{n-p}, i_*\omega_X^n).
\end{equation*}
Since $\omega_X^p\simeq \Hom_{\hol_X}(\mathit{\Omega}_X^{n-p}, \omega_X^n)$ by \cite[Proposition~3]{Barlet},
we will be done if we can show that the kernel of the map $b$ above consists of those 
homomorphisms $\mathit{\Omega}^{n-p}\to i_*\omega_X^n$ which in fact are homomorphisms $\mathit{\Omega}_X^{n-p}\to \omega_X^n$;
since $\J_X^0i_*\omega_X^n=0$, a homomorphism $\mathit{\Omega}^{n-p}\to i_*\omega_X^n$ is a homomorphism
$\mathit{\Omega}^{n-p}_X\to \omega_X^n$ if and only if it vanishes on $d\J_X^0\wedge\mathit{\Omega}^{n-p-1}$. 
To understand the map $b$ one can for instance use that $(\mathscr{C}^{N,\bullet},\debar)$,
where $\mathscr{C}^{N,\bullet}$ is the sheaf of germs of $(N,\bullet)$-currents in $\B$, 
is a resolution of $\mathit{\Omega}^{N}$ by stalk-wise injective sheaves. In fact, then
\begin{equation*}
\Ext^{\kappa}_{\hol}(\hol_X\otimes \mathit{\Omega}^{n-p}, \mathit{\Omega}^{N}) \simeq
\mathscr{H}^{\kappa}\big(\Hom_{\hol}(\mathit{\Omega}^{n-p}, \Hom_{\hol}(\hol_X, \mathscr{C}^{N,\bullet})),\debar\big)
\end{equation*}
and, since $\F=\hol_X\otimes\F$,
\begin{equation*}
\Ext^{\kappa}_{\hol}(\F, \mathit{\Omega}^{N}) \simeq
\mathscr{H}^{\kappa}\big(\Hom_{\hol}(\F, \Hom_{\hol}(\hol_X, \mathscr{C}^{N,\bullet})),\debar\big)
\end{equation*}
and the map $b$ is induced by restricting homomorphisms defined on $\mathit{\Omega}^{n-p}$ 
to the subsheaf $d\J_X^0\wedge \mathit{\Omega}^{n-p-1}$.
\end{proof}

It follows from Proposition~\ref{BHPprop} that $\omega_X^p$ is coherent, 
which, as mentioned above, also is proved in \cite{Barlet}.
That the vertical map in \eqref{diagramm} is an isomorphism can be seen as a version of/complement
to \cite[Lemma~4]{Barlet}. In fact, in view of \cite[Theorem~1.5]{ANoetherDual}, 
in our terminology that lemma says that $\omega_X^p$, via $i_*$, is isomorphic to the sheaf of germs of
Coleff-Herrera currents $\mu$ in $\B$ of bidegree $(\kappa+p,\kappa)$ such that $\J_X^0\mu=0$ and 
$d\J_X^0\wedge\mu=0$, i.e., such that $\mu=i_*\tau$ for some $(p,0)$-current $\tau$ on $X$;
cf.\ the paragraph after the dimension principle in Section~\ref{PMsektion}.
On the other hand, that the vertical map in \eqref{diagramm} is an isomorphism means that $\omega_X^p$, via $i_*$, is
isomorphic to the sheaf of germs of
Coleff-Herrera currents $\mu$ in $\B$ of bidegree $(\kappa+p,\kappa)$ such that $\J_X^{n-p}\wedge\mu=0$. 

That the vertical map in \eqref{diagramm} is an isomorphism also implies that the map
\begin{equation}\label{snobb}
\omega^p_X\to \Hom_{\hol_X}(\Om_X^{n-p}, \omega_X^n), \quad \mu \mapsto (\varphi \mapsto \mu\wedge\varphi)
\end{equation}
is an isomorphism. In fact, it is clear that \eqref{snobb} is injective, 
and if $\lambda$ is a homomorphism $\Om_X^{n-p}\to\omega_X^n$,
then $i_*\circ\lambda$ is a homomorphism $\Om_X^{n-p}\to \CH^N_X$. Since
the vertical map in \eqref{diagramm} is an isomorphism there is a $\mu\in \omega_X^p$
such that $i_*\circ\lambda (\varphi)=i_*(\mu\wedge\varphi)$ and thus \eqref{snobb} is surjective.
One may construe $\Om_X^{n-p}$ in \eqref{snobb}
also as $\mathit{\Omega}_X^{n-p}$, cf.\ the proof above, and thus we recover \cite[Proposition~3]{Barlet}.
We notice that \cite[Proposition~3]{Barlet} implies that $\omega_X^p$ coincides with the differential
$p$-forms considered by Kersken in \cite{Kersken}; Proposition~\ref{BHPprop} is \cite[Korollar~6.2 (2)]{Kersken}.

%%%%%%%%%%%%%%%%%%%%%%%%%%%%%%%%%%%%%%%%%%%%%%%%%%%%%%%%%%%%%%%%%%%%%%%%%%%%%%%%%%%%%%%%%%%%%%%%%%%%%%%%%%%%%%%%%%%%%%%%%%%%

\section{Integral operators on an analytic subset}\label{intopsection}
Let $D\subset \C^N$ be a domain (not necessarily pseudoconvex at this point), let $k(\zeta,z)$ be an integrable
$(N,N-1)$-form in $D\times D$, and let $p(\zeta,z)$ be a smooth $(N,N)$-form in $D\times D$. Assume that
$k$ and $p$ satisfy the equation of currents 
\begin{equation}\label{currentKoppel}
\debar k(\zeta,z) = [\Delta^D] - p(\zeta,z)
\end{equation}
in $D\times D$, where $[\Delta^D]$ is the current of integration along the diagonal. Here,
and always when working in a product space, $\debar$, $\partial$, etc., and (bi)degree is with respect to all coordinates.
Applying \eqref{currentKoppel} to test forms $\psi(z)\wedge \varphi(\zeta)$ it is straightforward to verify that for
any compactly supported $(p,q)$-form $\varphi$ in $D$ one has the following Koppelman formula
\begin{equation*}
\varphi(z) = \debar_{z} \int_{D_{\zeta}} k(\zeta,z)\wedge \varphi(\zeta) + 
\int_{D_{\zeta}} k(\zeta,z)\wedge \debar\varphi(\zeta) + \int_{D_{\zeta}} p(\zeta,z)\wedge \varphi(\zeta).
\end{equation*}

In \cite{AIntRepI} Andersson introduced a very flexible method of producing solutions to \eqref{currentKoppel}.
Let $\eta=(\eta_1,\ldots,\eta_N)$ be a holomorphic tuple in $D\times D$ that defines the diagonal and let 
$\Lambda_{\eta}$ be the exterior algebra spanned by $\Lambda^{0,1}T^*(D \times D)$
and the $(1,0)$-forms $d\eta_1,\ldots,d\eta_N$. On forms with values in $\Lambda_{\eta}$ interior multiplication
with $2\pi i \sum\eta_j \partial/\partial \eta_j$, denoted $\delta_{\eta}$, is defined; set $\nabla_{\eta}=\delta_{\eta}-\debar$.

Let $s$ be a smooth $(1,0)$-form in $\Lambda_{\eta}$ such that $|s|\lesssim |\eta|$
and $|\eta|^2\lesssim |\delta_{\eta}s|$ and let $B=\sum_{k=1}^N s\wedge (\debar s)^{k-1}/(\delta_{\eta}s)^k$.
It is proved in \cite{AIntRepI} that then $\nabla_{\eta} B = 1-[\Delta^D]$. Identifying terms of top degree
we see that $\debar B_{N,N-1} = [\Delta^D]$ and we have found a solution to \eqref{currentKoppel}.
For instance, if we take $s=\partial |\zeta-z|^2$ and $\eta=\zeta-z$, then the resulting $B$ is sometimes called
the full Bochner-Martinelli form and the term of top degree is the classical Bochner-Martinelli kernel. 

A smooth section $g(\zeta,z)=g_{0,0}+\cdots +g_{N,N}$ of $\Lambda_{\eta}$, where the subscript means bidegree, defined for $z\in D'\subset D$
and $\zeta\in D$, such that $\nabla_{\eta} g=0$ and $g_{0,0}\restriction_{\Delta^D} = 1$ 
is called a \emph{weight} 
with respect to $z\in D'$. It follows that $\nabla_{\eta} (g\wedge B) = g-[\Delta^D]$ and, identifying 
terms of bidegree $(N,N-1)$, we get that
\begin{equation}\label{gulp}
\debar (g\wedge B)_{N,N-1} = [\Delta^D] - g_{N,N}
\end{equation} 
in $D_{\zeta}\times D'_z$ and hence another solution to \eqref{currentKoppel}.
% ; if $g$ is matrix-valued 
% we replace $[\Delta^D]$ by $\textrm{Id}\otimes [\Delta^D]$.
If $D$ is pseudoconvex and $K$ is a holomorphically convex compact subset, then one can find a weight $g$ with
respect to $z$ in some neighborhood $D'\Subset D$ of $K$ such that $z\mapsto g(\zeta,z)$ is holomorphic 
in $D'$ and $\zeta \mapsto g(\zeta,z)$ has compact support in $D$; see, e.g., \cite[Example~2]{AIntRepII} 
or \cite[Example~5.1]{AS} in case $D=\B$.
We will also have use for weights with values in a certain type of vector bundle, cf.\ \cite{GSS} and \cite{AIntRepII}. Let $V\to D$
be a vector bundle, let $\pi_{\zeta}\colon D_{\zeta}\times D_z\to D_{\zeta}$ and $\pi_z\colon D_{\zeta}\times D_z\to D_{z}$
be the natural projections and set $V_z\otimes V_{\zeta}^*:=\pi_z^* V \otimes \pi^*_{\zeta} V^*$.
Then a weight may take values in $V_z\otimes V_{\zeta}^*\simeq \textrm{Hom}(V_{\zeta},V_z)$; 
it should satisfy the same properties but with
the condition $g_{0,0}\restriction_{\Delta^D} = 1$ replaced by $g_{0,0}\restriction_{\Delta^D} = \textrm{Id}_V$.
If $g$ is a weight with values in $V_z\otimes V_{\zeta}^*$ then \eqref{gulp} holds with $[\Delta^D]$ replaced
by $\textrm{Id}_V\otimes [\Delta^D]$.

\medskip

Let $\tilde{X}$ be an analytic subset of pure codimension $\kappa$ of a neighborhood of $\overline{D}$,
where $D$ now is assumed to be strictly pseudoconvex, and set $X=\tilde{X}\cap D$.
Let \eqref{eq:resol} be a free resolution of $i_*\Om_X^p$ in $D$ and let $U=U(\zeta)$ and $R=R(\zeta)$ be the 
associated currents (for some choice of Hermitian metrics on the $E_k$'s). 
% Let us, for convenience,
% fix holomorphic frames for the $E_k$'s.
% Let $\pi_1\colon D_{\zeta}\times D_z\to D_{\zeta}$ and $\pi_2\colon D_{\zeta}\times D_z\to D_{z}$
% be the natural projections. Via the frames we may identify 
% $\pi_1^* E_k \simeq D_{\zeta}\times D_z\times \C^{\textrm{rank}\, E_k} \simeq \pi^*_2 E_k$
% and we let for the rest of this section (unless otherwise stated) $E_k$ denote the bundle
% $D_{\zeta}\times D_z\times \C^{\textrm{rank}\, E_k}$ and we view all morphisms $E_k\to E_{\ell}$
% as matrices.
Let 
$E_k^z:=\pi_z^*E_k$ and $E_k^{\zeta}:=\pi_{\zeta}^*E_k$.
One can find Hefer morphisms $H_k^{\ell}=H_k^{\ell}(\zeta,z)$, which depending holomorphically on $(\zeta,z)\in D\times D$ and
are $\textrm{Hom}(E_k^{\zeta},E_\ell^z)$-valued $(k-\ell,0)$-forms such that
\begin{equation*}
H_k^{k}\restriction_{\Delta_D} = \textrm{Id}_{E_k} \quad \textrm{and} \quad
\delta_{\eta} H_k^{\ell} = H_{k-1}^{\ell}f_k - f_{\ell+1}(z)H_k^{\ell+1}, \,\, k>\ell,
\end{equation*}
where $f_k=f_k(\zeta)$; see \cite[Proposition~5.3]{AIntRepII}.
Let $F=F(\zeta)$ be a holomorphic tuple 
such that $X=\{F=0\}$ and set $\chi^{\epsilon}:=\chi(|F|^2/\epsilon)$; we regularize $U$ and $R$
as in Section~\ref{prelim} so that $U^{\epsilon}:=\chi^{\epsilon}u$ and 
\begin{equation*}
R^{\epsilon}:=\textrm{Id}_E - \nabla U^{\epsilon}=(1-\chi^{\epsilon})\textrm{Id}_E + \debar\chi^{\epsilon}\wedge u.
\end{equation*}
We write $U^{\epsilon}_k$ and $R^{\epsilon}_k$ for the parts of $U^{\epsilon}$ and $R^{\epsilon}$ that take 
values in $\textrm{Hom}(E_0,E_k)$ and we define
\begin{equation*}
G^{\epsilon}:=
\sum_{k\geq 0} H^0_kR^{\epsilon}_k + f_1(z)\sum_{k\geq 1}H^1_{k}U^{\epsilon}_k,
\end{equation*}
which one can check is a weight with values in $\textrm{Hom}(E_0^{\zeta},E_0^z)$.

Letting $g$ be any scalar-valued weight with respect to, say, $z\in D'\subset D$
it follows that $G^{\epsilon}\wedge g$ is a $\textrm{Hom}(E_0^{\zeta},E_0^z)$-valued
weight and \eqref{gulp} holds with $g$ replaced by $G^{\epsilon}\wedge g$ and $[\Delta^D]$ replaced by
$\textrm{Id}_{E_0}\otimes [\Delta^D]$.
Let $\nabla^z=\oplus_jf_j(z)-\debar$ and let $\nabla_{\textrm{End}}^z$ be the corresponding endomorphism-valued operator.
Then, recalling that $\nabla_{\textrm{End}}^z R(z)=0$ and noticing that 
$\nabla_{\textrm{End}}^z (G^{\epsilon}\wedge g\wedge B)=-\debar (G^{\epsilon}\wedge g\wedge B)$
since $f(z)\restriction_{E_0}=0$, we get
\begin{equation}\label{gulp2}
-\nabla_{\textrm{End}}^z \big(R(z)\wedge dz\wedge (G^{\epsilon}\wedge g\wedge B)_{N,N-1}\big) \quad \quad \quad
\quad \quad \quad \quad \quad \quad \quad \quad \quad
\end{equation}
\begin{equation*}
\quad \quad \quad \quad \quad \quad \quad \quad \quad
=R(z)\wedge dz\wedge [\Delta^D] - R(z)\wedge dz \wedge (G^{\epsilon}\wedge g)_{N,N}.
\end{equation*}
Notice that $R(z)\wedge [\Delta^D]$ and $R(z)\wedge B$ are well-defined; they are simply tensor products of currents
since $z$ and $\zeta-z$ are independent variables on $D\times D$.
Since $R(z)f_1(z)=0$, \eqref{gulp2} becomes
\begin{equation}\label{tv}
-\nabla_{\textrm{End}}^z \big(R(z)\wedge dz\wedge (HR^{\epsilon}\wedge g\wedge B)_{N,N-1}\big) \quad \quad \quad
\quad \quad \quad \quad \quad \quad \quad \quad \quad
\end{equation}
\begin{equation*}
\quad \quad \quad \quad \quad \quad \quad \quad \quad
=R(z)\wedge dz\wedge [\Delta^D] - R(z)\wedge dz \wedge (HR^{\epsilon}\wedge g)_{N,N},
\end{equation*}
where $HR^{\epsilon}:=\sum_{k\geq 0} H^0_kR^{\epsilon}_k$.
Let $\iota\colon X\simeq\Delta^X\hookrightarrow X\times X$ be the diagonal embedding and let 
$\mathfrak{i}\colon X\times X\hookrightarrow D\times D$ be the inclusion. By Proposition~\ref{fundprop}
we have 
\begin{equation}\label{tv2}
\mathfrak{i}_*\iota_* \omega = R(z)\wedge dz\wedge [\Delta^D],
\end{equation}
where $\omega$ is the $n-p$-structure form corresponding to $R$. 

Consider now the term $(HR^{\epsilon}\wedge g)_{N,N}$. Noticing that $R^{\epsilon}$ contains no $d\eta_j$ we see that 
\begin{equation}\label{sent1}
(HR^{\epsilon}\wedge g)_{N,N}= \tilde{p}(\zeta,z)\wedge R^{\epsilon}\wedge d\eta,
\end{equation}
for some $\textrm{Hom}(E^{\zeta},E_0^z)$-valued form $\tilde{p}(\zeta,z)$ that is smooth for $(\zeta,z)\in D\times D'$; 
if $g$ is chosen holomorphic in $z$ (respectively $\zeta$), then $\tilde{p}$ is holomorphic in $z$ (respectively $\zeta$).
To further reveal the structure of $\tilde{p}$, let $\varepsilon_1,\ldots,\varepsilon_N$ be a frame for an auxiliary 
trivial vector bundle $F\to D\times D$, replace each occurrence of $d\eta_j$ in $H$ and $g$ by $\varepsilon_j$,
and denote the result by $\hat{H}$ and $\hat{g}$. We get
\begin{equation}\label{sent2}
\tilde{p}(\zeta,z)\wedge R^{\epsilon}\wedge \varepsilon =  (\hat{H}R^{\epsilon}\wedge \hat{g})_{N,N} =
\sum_{k\geq 0} \hat{H}^0_kR_k^{\epsilon}\wedge \hat{g}_{N-k,N-k} =
\sum_{k\geq 0} \tilde{p}_k(\zeta,z)\wedge R^{\epsilon}_k\wedge \varepsilon,
\end{equation}
where $\tilde{p}_k(\zeta,z)=\pm \epsilon^*\lrcorner \hat{H}^0_k\wedge \hat{g}_{N-k,N-k}$ is a smooth 
$(0,N-k)$-form in $D\times D'$ with values in $\textrm{Hom}(E_k^{\zeta},E_0^z)$; it is holomorphic in $z$ (or $\zeta$)
if $g$ is chosen so. For degree reasons it follows that 
\begin{equation}\label{bibblan}
R(z)\wedge dz\wedge (HR^{\epsilon}\wedge g)_{N,N} = 
R(z)\wedge dz \wedge \sum_{k\geq 0}\tilde{p}_k(\zeta,z)\wedge R^{\epsilon}_k\wedge d\zeta.
\end{equation}
Since $R(z)\wedge R$ is well-defined (as a tensor product) we may set $\epsilon=0$ in \eqref{bibblan}
and since $R=R_{\kappa}+R_{\kappa+1}+\cdots$ we then sum only over $k\geq \kappa$. In view of 
Proposition~\ref{fundprop} it follows that 
\begin{equation}\label{tv3}
\lim_{\epsilon\to 0} R(z)\wedge dz\wedge (HR^{\epsilon}\wedge g)_{N,N} =
\mathfrak{i}_* \omega(z)\wedge p(\zeta,z),
\end{equation} 
where
\begin{equation*}
p(\zeta,z):=\sum_{k\geq \kappa} \mathfrak{i}^*\tilde{p}_k(\zeta,z)\wedge\omega_{k-\kappa}(\zeta)=
\sum_{k\geq \kappa} \pm \mathfrak{i}^*\big(\varepsilon^*\lrcorner \hat{H}^0_k\wedge \hat{g}_{N-k,N-k}\big)
\wedge\omega_{k-\kappa}(\zeta).
\end{equation*}
We here, and in the following, view $\tilde{p}_k$ not as $(0,N-k)$-form with values in $\textrm{Hom}(E_k^{\zeta},E_0^z)$ but
as a $(p,N-k)$-form with values in $(E_k^{\zeta})^*$; cf.\ the paragraph preceding Proposition~\ref{fundprop}. Thus,
$p(\zeta,z)$ is a scalar valued almost semi-meromorphic current on $X\times X'$ of bidegree $(n,n)$
such that $z\mapsto p(\zeta,z)$ is, or rather, has a natural extension that is smooth in $D$ 
(or holomorphic if $z\mapsto g(\zeta,z)$ is); notice that $p(\zeta,z)$ has degree $p$ in $dz_j$ and degree
$n-p$ in $d\zeta_j$.  

We proceed in an analogous way with the current $R(z)\wedge dz\wedge (HR^{\epsilon}\wedge g\wedge B)_{N,N-1}$
and we get, cf.\ \eqref{bibblan}, that
\begin{equation}\label{bibblan2}
R(z)\wedge dz\wedge (HR^{\epsilon}\wedge g\wedge B)_{N,N-1} = 
R(z)\wedge dz\wedge \sum_{j\geq 0} \tilde{k}_j(\zeta,z)\wedge R^{\epsilon}_j\wedge d\zeta, 
\end{equation}
where $\tilde{k}_j(\zeta,z):=\pm\varepsilon^*\lrcorner \hat{H}^0_j\wedge (\hat{g}\wedge \hat{B})_{N-j,N-j-1}$
is a $(0,N-j-1)$-form with values in $\textrm{Hom}(E_j^{\zeta},E_0^z)$. From Section~\ref{prelim} we know that the limit
as $\epsilon\to 0$ of \eqref{bibblan2} exists and yields a pseudomeromorphic current in $D\times D'$. Moreover,
precisely as in \cite[Lemma 5.2]{AS} one shows that 
\begin{equation*}
\lim_{\epsilon\to 0} R(z)\wedge dz\wedge (HR^{\epsilon}\wedge g\wedge B)_{N,N-1}=
\lim_{\epsilon\to 0}  R(z)\wedge dz\wedge (HR\wedge g\wedge B^{\epsilon})_{N,N-1},
\end{equation*}
where $B^{\epsilon}:=\chi(|\eta|^2/\epsilon) B$, holds
in the sense of current on $(D\setminus X_{sing})\times (D'\setminus X_{sing})$.
In view of \eqref{bibblan2} and Proposition~\ref{fundprop} we thus get
\begin{equation}\label{gurka}
\lim_{\epsilon\to 0} R(z)\wedge dz\wedge (HR^{\epsilon}\wedge g\wedge B)_{N,N-1} =
\lim_{\epsilon\to 0}  \chi(|\eta|^2/\epsilon) \mathfrak{i}_* \omega(z)\wedge k(\zeta,z)
\end{equation}
in $(D\setminus X_{sing})\times (D'\setminus X_{sing})$, where
\begin{equation}\label{k}
k(\zeta,z):=\sum_{j\geq \kappa} \mathfrak{i}^*\tilde{k}_j(\zeta,z)\wedge \omega_{j-\kappa}(\zeta)=
\pm\sum_{j\geq \kappa} \mathfrak{i}^* \big(\varepsilon^*\lrcorner \hat{H}^0_j\wedge (\hat{g}\wedge \hat{B})_{N-j,N-j-1}\big)
\wedge \omega_{j-\kappa}(\zeta).
\end{equation}
As with $\tilde{p}_j(\zeta,z)$, we here and in the following view $\tilde{k}_j(\zeta,z)$ as a $(p,N-j-1)$-form
with values in $(E_j^{\zeta})^*$ so that $k(\zeta,z)$ becomes a scalar valued almost semi-meromorphic 
$(n,n-1)$-current on $X\times X'$;
the degree in $dz_j$ being $p$ and the degree in $d\zeta_j$ being $n-p$.
Recall that $B_{\ell,\ell-1}=s\wedge(\debar s)^{\ell-1}/(\delta_{\eta}s)^{\ell}$ and that $|s|\lesssim |\eta|$ and
$|\eta|^2\lesssim |\delta_{\eta}s|$. Since $\hat{B}_{\ell,\ell-1}$, $\ell=1,\cdots,n$ are the only components of 
$\hat{B}$ that enter in the expression for $k(\zeta,z)$ it follows that $k(\zeta,z)$ is integrable on 
$X_{reg}\times X'_{reg}$. Hence, the limit on the right hand side of \eqref{gurka} is just the locally integrable
form $k(\zeta,z)\wedge\omega(\zeta)$ on $X_{reg}\times X'_{reg}$. From \eqref{tv}, \eqref{tv2}, \eqref{tv3}, and
\eqref{gurka} we thus see that
\begin{equation}\label{tv4}
-\nabla \omega(z)\wedge k(\zeta,z) = \iota_*\omega - \omega(z)\wedge p(\zeta,z)
\end{equation}
as currents on $X_{reg}\times X'_{reg}$, where $\nabla$ here means the endomorphism-version
of $f(z)\restriction_X-\debar$. Since $R$ is $\nabla_{\textrm{End}}$-closed it follows that $\omega(z)$
is $\nabla$-closed and so the left hand side of \eqref{tv4} equals $\omega(z)\wedge \debar k(\zeta,z)$. By 
Lemma~\ref{divlemma} we have thus proved
\begin{proposition}
In $X_{reg}\times X'_{reg}$ we have that $\debar k(\zeta,z) = [\Delta^X] - p(\zeta,z)$ as currents.
\end{proposition}

The following technical lemma corresponds to \cite[Lemma~6.4]{AS}; cf.\ also \cite[Proposition~4.3 (ii)]{RSWSerre}.
It is a statement on $X^{\nu+1}:=X\times\cdots\times X$ ($\nu+1$ factors); 
$X_{z^j}$ refers to the $j^{\text{th}}$ factor and $z^j$
are points on $X_{z^j}$.

\begin{lemma}\label{ASlemma}
Let $\omega$ be any $n-p$-structure form and let
$k_j(z^{j-1},z^j)$, $j=1,\ldots,\nu$, be given by \eqref{k} for possibly different choices of $H$'s, $g$'s, $B$'s, and 
$n-p$-structure forms $\omega$'s. Then
\begin{equation}\label{T}
T:=\omega(z^{\nu})\wedge k_{\nu}(z^{\nu-1},z^{\nu})\wedge k_{\nu-1}(z^{\nu-2},z^{\nu-1})\wedge \cdots
\wedge k_{1}(z^{0},z^{1}) 
\end{equation}   
is an almost semi-meromorphic current on $X^{\nu+1}$. If $h=h(z^j)$ is a generically non-vanishing holomorphic
tuple on $X_{z^j}$ then  $\debar\chi(|h|^2/\epsilon)\wedge T\to 0$ as $\epsilon\to 0$. 
\end{lemma}

\subsection{The integral operators $\K$ and $\Proj$ on $(p,*)$-forms.}\label{KoPsection}
In order to construct the integral operators $\K$ we choose the weight $g$ in the definitions of 
$p(\zeta,z)$ and $k(\zeta,z)$ to be a weight with respect to $z\in D'\Subset D$
such that $\zeta\mapsto g(\zeta,z)$ has compact support in $D$.
Let $\varphi$ be a pseudomeromorphic $(p,q)$-current on $X$. In view of Section~\ref{PMsektion},
$k(\zeta,z)\wedge\varphi(\zeta)$ and $p(\zeta,z)\wedge\varphi(\zeta)$ are well-defined 
pseudomeromorphic currents in $X_{\zeta}\times X_z'$, where $X'=X\cap D'$. Let 
$\pi^z\colon X_{\zeta}\times X_z \to X_z$ be the natural projection and set
\begin{equation}\label{KoP}
\K\varphi(z):=\pi^z_* \, k(\zeta,z)\wedge\varphi(\zeta),\quad
\Proj\varphi(z):=\pi^z_* \, p(\zeta,z)\wedge\varphi(\zeta).
\end{equation}
Since $\zeta\mapsto g(\zeta,z)$ has compact support in $D$ it follows that
$\K\varphi$ and $\Proj\varphi$ are well-defined pseudomeromorphic currents in $X'$. 
Notice that $\Proj\varphi$ has a natural smooth extension to $D'$ since $z\mapsto p(\zeta,z)$ has;
notice also that if $\varphi$ has the SEP then $\K\varphi$ has the SEP in view of Section~\ref{PMsektion}.
Moreover, as in \cite[Lemma~6.1]{AS} one shows that if $\varphi=0$ in a neighborhood
of a point $x\in X'$, or if $\varphi$ is smooth in a neighborhood of $x$ and $x\in X_{reg}'$, 
then $\K\varphi$ is smooth in a neighborhood of $x$.

If $\varphi$ is a pseudomeromorphic $(p,q)$-current with compact support in $X$, then 
one can choose any weight $g$ in the definitions of $k(\zeta,z)$ and $p(\zeta,z)$
and define $\K\varphi$ and $\Proj\varphi$ by \eqref{KoP}; the outcome has the same 
general properties.

The following proposition is proved in the same way as \cite[Proposition~6.3]{AS}.

\begin{proposition}\label{kprop1}
Let $\varphi\in \W^{p,q}(X)$, let $\omega$ be the $n-p$-structure form that enters in the definitions
of $k(\zeta,z)$ and $p(\zeta,z)$, and assume that $\debar (\omega\wedge\varphi)$ has the SEP. 
Let $g$ be a weight with respect to $z\in D'\subset D$.
If either $g$ has compact support in $D_{\zeta}$ or $\varphi$ has compact support in $X$ then
$\varphi = \debar\K\varphi + \K(\debar\varphi) + \Proj\varphi$
as currents on $X'_{reg}$.
\end{proposition}

Notice that the condition that $\debar(\omega\wedge\varphi)$ has the SEP implies that $\debar\varphi$
has the SEP. In fact, from Section~\ref{PMsektion} we know that 
$\omega\wedge\varphi$ has the SEP and so, in view of Lemma~\ref{obslemma},
$\debar(\omega\wedge\varphi)$ has the SEP if and only if 
$\debar\chi(|h|^2/\epsilon)\wedge\omega\wedge\varphi\to 0$ for all generically non-vanishing $h$. In particular,
$\debar\chi(|h|^2/\epsilon)\wedge\omega_0\wedge\varphi\to 0$ and so, by Lemma~\ref{divlemma},
$\debar\chi(|h|^2/\epsilon)\wedge\varphi\to 0$. By Lemma~\ref{obslemma} again we conclude that $\debar\varphi$ has the SEP.

From Proposition~\ref{kprop1} it is easy to prove the following residue criterion
for a meromorphic $p$-form to be strongly holomorphic. Recall the operator $\nabla=\oplus_jf_j - \debar$.
attached to \eqref{eq:resol}.

\begin{theorem}\label{rescrit}
Let $X$ be a pure $n$-dimensional analytic subset of some neighborhood of the closure of a strictly pseudoconvex domain
$D\in\C^N$ and let $\omega$ be an
$n-p$-structure form on $X\cap D$ corresponding to a resolution \eqref{eq:resol} of $\Om_X^p$. 
Then a meromorphic $p$-form $\varphi$ on $X\cap D$
is strongly holomorphic if and only if 
\begin{equation}\label{rescrit1}
\nabla(\omega\wedge\varphi)=0.
\end{equation}
Moreover, if \eqref{rescrit1} holds, $D'\Subset D$, and $\Proj$ is an integral operator
constructed using $\omega$ and a weight $g(\zeta,z)$ such that
$z\mapsto g(\zeta,z)$ is holomorphic in $D'$ and $\zeta\mapsto g(\zeta,z)$ has compact support in $D$,
then $\Proj\varphi$ is a holomorphic extension of $\varphi_{\restriction_{X\cap D'}}$ to $D'$.
\end{theorem}

\begin{proof}
Notice first that if $\varphi$ is strongly holomorphic then \eqref{rescrit1} holds since $\nabla \omega=0$.

For the converse, notice that $\omega\wedge\varphi$ has the SEP so that 
$\chi(|h|^2/\epsilon)\omega\wedge\varphi\to \omega\wedge\varphi$ for all generically non-vanishing $h$.
Hence, if \eqref{rescrit1} holds, we get 
\begin{equation*}
0=\nabla(\omega\wedge\varphi) = \lim_{\epsilon\to 0} \nabla \big(\chi(|h|^2/\epsilon)\omega\wedge\varphi\big) =
-\lim_{\epsilon\to 0}\debar \chi(|h|^2/\epsilon)\wedge \omega\wedge\varphi.
\end{equation*}
for all such $h$.
From Lemma~\ref{obslemma} it thus follows that $\debar(\omega\wedge\varphi)$ has the SEP.
From the paragraph after Proposition~\ref{kprop1} it then follows that $\debar\varphi$ has the SEP and since $\varphi$
is holomorphic generically we see that $\debar\varphi=0$. By Proposition~\ref{kprop1} we get that
$\varphi=\Proj\varphi$ on $X_{reg}\cap D'$. However, both $\varphi$ and $\Proj\varphi$ have the SEP so this holds on $X\cap D'$. 
\end{proof}

% We have already seen one application of this result, namely Proposition~\ref{sumarum}.
% Let us try to shed some further light on the criterion. Assume for simplicity that $p=0$
% and that $\varphi=g/h$, where $h$ is generically non-vanishing on $X$.
% In view of Proposition~\ref{fundprop}, \eqref{rescrit1} is equivalent to $\nabla(\varphi R)=0$,
% which is Andersson's formulation \cite{Acrit} of the criterion. Assume in addition that
% $X=\{f_1=\cdots=f_{\kappa}=0\}$ is a reduced complete intersection. Then
% $R$ is the Coleff-Herrera product $\debar(1/f_1)\wedge\cdots\wedge\debar (1/f_{\kappa})$ (times some
% frame element), cf.\ Example~\ref{koszulex}. In this case \eqref{rescrit1} is thus equivalent to
% $g \debar(1/h)\wedge \debar (1/f_1)\wedge \cdots \wedge \debar (1/f_{\kappa}) =0$,
% which holds if and only if $g$ is in the ideal $\langle h,f_1,\ldots,f_{\kappa}\rangle$,
% i.e., if and only if $g$ is divisible by $h$ on $X$. This is the philosophy of Tsikh's result \cite{Tsikh}.

Theorem~\ref{rescrit} gives the following geometric criterion for a meromorphic $p$-form to be strongly holomorphic.

\begin{proposition}\label{extensionprop}
Let $X$ be a pure $n$-dimensional reduced complex space
and let $\varphi$ be a meromorphic $p$-form on $X$ with pole set $P_{\varphi}\subset X$. Suppose that
(i) $\textrm{codim}_X P_{\varphi} \geq 2$, and that (ii) 
$\textrm{codim}_X S_{n-k}(\Om_X^p)\cap P_{\varphi} \geq k +2$ for $k\geq 1$.
% \begin{itemize}
% \item[(i)] $\textrm{codim}_X P_{\varphi} \geq 2$,
% \item[(ii)] $\textrm{codim}_X S_{n-k}(\Om_X^p)\cap P_{\varphi} \geq k +2$ for $k\geq 1$.
% \end{itemize}
Then $\varphi$ is strongly holomorphic.
\end{proposition}

\begin{proof}
Since $\Om_X^p$ is torsion free a strongly holomorphic extension of $\varphi$, if such exist, is unique.
Therefore the statement of the proposition is local and we may assume that $X$ is an analytic subset of a 
neighborhood of $\overline{\B}\subset\C^N$.
Let $\omega=\omega_0+\cdots$ be an $n-p$-structure form on $X\cap \B$. 
By Theorem~\ref{rescrit} we need to show that $\nabla(\omega\wedge\varphi)=0$. Since $\omega$ and $\varphi$ are 
almost semi-meromorphic we have 
$\pm\omega\wedge\varphi=\varphi\wedge\omega=\lim_{\epsilon\to 0}\chi(|h|^2/\epsilon) \varphi\wedge\omega$,
where $h$ is a generically non-vanishing holomorphic function such that $\{h=0\}\supset P_{\varphi}$. Thus, since
$\nabla\omega=0$, we see that 
$\nabla(\omega\wedge\varphi)=\pm\lim_{\epsilon\to 0} \debar\chi(|h|^2/\epsilon)\wedge \varphi\wedge\omega$
and so we need to show that
\begin{equation}\label{thai}
\lim_{\epsilon\to 0} \debar\chi(|h|^2/\epsilon)\wedge\varphi\wedge \omega_{\ell} =0
\end{equation}
for $\ell=0,1,2,\ldots$. 
For $\ell=0$ the left hand side of \eqref{thai} is a pseudomeromorphic $(n,1)$-current on $X$
with support contained in $P_{\varphi}$; hence it vanishes by the dimension principle and assumption (i).

Recall from Section~\ref{AWcurrsektion} the sets $Z_k$ associated
with a resolution \eqref{eq:resol} of $\Om^p_X$ and
that $S_{N-k}(\Om^p_X)=Z_{k}$. Assumption (ii) is thus equivalent to 
$\textrm{codim}\, Z_k\cap P_{\varphi}\geq k+2$ for $k\geq N-n+1$. Now, assume that \eqref{thai} holds for $\ell=m$.
Since, by Proposition~\ref{fundprop} (ii),
$\omega_{m+1}$ is a smooth form times $\omega_m$ outside of $Z_{m+1}$
it follows that for $\ell=m+1$ the left hand side of \eqref{thai} is a pseudomeromorphic
$(n,m+2)$-current with support contained in $Z_{m+1}\cap P_{\varphi}$. Thus,
\eqref{thai} holds for $\ell=m+1$ by assumption (ii) and the dimension principle.
\end{proof}

\subsection{The integral operators $\check{\K}$ and $\check{\Proj}$ on $(n-p,*)$-forms.}
A general integral operator $\check{\K}$ is constructed by choosing the weight $g$
in the definitions of $k(\zeta,z)$ and $p(\zeta,z)$ to be a weight
with respect to $\zeta\in D'\Subset D$ such that $z\mapsto g(\zeta,z)$ has compact support 
in $D$. Let $\psi$ be a pseudomeromorphic $(n-p,q)$-current on $X$. In the same way as above
$k(\zeta,z)\wedge\varphi(z)$ and $p(\zeta,z)\wedge\varphi(z)$ are well-defined 
pseudomeromorphic currents in $X_{\zeta}'\times X_z$ and we set
\begin{equation*}
\check{\K} \psi (\zeta) := \pi^{\zeta}_* \, k(\zeta,z)\wedge\varphi(z), \quad
\check{\Proj} \psi (\zeta) := \pi^{\zeta}_* \, p(\zeta,z)\wedge\varphi(z),
\end{equation*}
which become pseudomeromorphic currents on $X'$. Notice that $\check{\Proj} \psi$
has the SEP if $\psi$ has, and moreover,
is of the form $\sum_{\ell\geq 0} A_{\ell}(\zeta)\wedge \omega_{\ell}(\zeta)$,
where $A_{\ell}$ is a smooth form with values in $E_{\kappa+\ell}^*$; if $g$ is chosen 
so that $\zeta\mapsto g(\zeta,z)$ is holomorphic then the $A_{\ell}$ are holomorphic.
The current $\check{\K}\psi$ has the SEP if $\psi$ has, and it has the
form $\sum_{\ell\geq 0} C_{\ell}(\zeta)\wedge\omega_{\ell}(\zeta)$,
where the $C_{\ell}$ take values in $E_{\kappa+\ell}^*$ and are: i) smooth close to $x\in X'$ if $\psi=0$  
close to $x$, and ii) smooth close to $x\in X'_{reg}$ if $\psi$ is smooth close to $x$.

As for $\K$ and $\Proj$, if $\psi$ happens to have compact support in $X$ then any weight $g$ 
may be used to define $\check{\K} \psi$ and $\check{\Proj} \psi$.

\begin{proposition}\label{kprop2}
Let $\psi\in \W^{n-p,q}(X)$, assume that $\debar\psi\in\W^{n-p,q+1}(X)$, and let $g$ be a weight
with respect to $\zeta\in D'\subset D$. If either $g$ has compact support in $D_z$ or $\psi$ has 
compact support in $X$ then $\psi=\debar\check{\K}\psi + \check{\K}(\debar\psi) + \check{\Proj}\psi$
as currents on $X'_{reg}$.
\end{proposition}

This is proved in the same way as \cite[Proposition~3.1]{RSWSerre}.

%%%%%%%%%%%%%%%%%%%%%%%%%%%%%%%%%%%%%%%%%%%%%%%%%%%%%%%%%%%%%%%%%%%%%%%

\section{The sheaves $\A_X^{p,q}$ and $\Bsheaf_X^{n-p,q}$}\label{AoBsection}
\subsection{The sheaves $\A_X^{p,\bullet}$}\label{Asubsection}
Let $X$ be a reduced complex space of pure dimension $n$.
Following \cite[Definition~7.1]{AS} we say that a $(p,q)$-current $\varphi$ on $X$ on an open subset $U\subset X$
is a section of $\A_X^{p,q}$ over $U$ if for every $x\in U$ the germ $\varphi_x$ can be written as a finite sum of terms
\begin{equation}\label{A}
\xi_{\nu} \wedge \K_{\nu} (\cdots \xi_2\wedge \K_2 (\xi_1\wedge\K_1(\xi_0)) \cdots ),
\end{equation}
where $\xi_0$ is a smooth $(p,*)$-form and the $\xi_j$, $j\geq 1$, are smooth $(0,*)$-forms such that 
$\xi_j$ has support where $z\mapsto k_j(\zeta,z)$ is defined.

\begin{proposition}\label{Aprop1}
The sheaf $\A_X^{p,q}$ has the following properties:
\begin{itemize}
\item[(i)] $\E_X^{p,q}\subset \A_X^{p,q}\subset \W_X^{p,q}$ and $\oplus_q \A_X^{p,q}$ is a module over
$\oplus_q\E_X^{0,q}$,

\item[(ii)] $\A_{X_{reg}}^{p,q}=\E_{X_{reg}}^{p,q}$,

\item[(iii)] for any operator $\K$ on $(p,*)$-forms as in Section~\ref{KoPsection}
$\K\colon \A_X^{p,q} \to \A_X^{p,q-1}$,

\item[(iv)] if $\varphi$ is a section of $\A_X^{p,q}$ and $\omega$ is any $n-p$-structure form,
then $\debar (\omega\wedge\varphi)$ has the SEP.
\end{itemize}
\end{proposition}
\begin{proof}
(i), (ii), and (iii) are immediate from the definition of $\A_X^{p,q}$ and the general properties of the 
$\K$-operators in Section~\ref{KoPsection}. To prove (iv) we may assume that $\varphi$ is of the form
\eqref{A}. Then $\omega\wedge\varphi$ is a push-forward of $T\wedge \xi$, where $T$ is of the form \eqref{T}
and $\xi$ is a smooth form on $X^{\nu+1}$. Choosing $h=h(z^{\nu})$ in Lemma~\ref{ASlemma}
it follows that $\debar\chi(|h|^2/\epsilon)\wedge \omega\wedge\varphi\to 0$ as $\epsilon\to 0$
and so, by Lemma~\ref{obslemma}, $\debar (\omega\wedge\varphi)$ has the SEP.
\end{proof}

\begin{proof}[Proof of Theorem~\ref{main2}]
Let $D''\Subset D$ be a strictly pseudoconvex neighborhood of $\overline{D}'$
and carry out the construction of $k(\zeta,z)$ and $p(\zeta,z)$ in Section~\ref{intopsection}
in $D''\times D''$ using a weight $g(\zeta,z)$ with respect to $z\in D'$ such that
$z\mapsto g(\zeta,z)$ is holomorphic in $D'$ and $\zeta\mapsto g(\zeta,z)$ has compact support 
in $D''$. Notice that then $\Proj\varphi$ is holomorphic and that $g$, and hence also $p(\zeta,z)$, has bidegree $(*,0)$ 
in the $z$-variables
so that $\Proj\varphi=0$ if $\varphi$ has bidegree $(p,q)$ with $q\geq 1$. Let $\varphi\in\A^{p,q}(X)$.
By Proposition~\ref{Aprop1} (iv), $\debar (\omega\wedge\varphi)$ has the SEP and so Proposition~\ref{kprop1} shows that
\begin{equation}\label{eq:main2}
\varphi = \debar \K \varphi + \K(\debar\varphi) + \Proj\varphi
\end{equation}
in the sense of currents on $X'_{reg}$. 
Now, $\K\varphi\in\A^{p,q-1}(X')$ by Proposition~\ref{Aprop1} (iii). Hence, by Proposition~\ref{Aprop1} (iv)
and the comment after Proposition~\ref{kprop1}, $\debar\K\varphi$ has the SEP. In the same way $\debar\varphi$
has the SEP and so $\K(\debar\varphi)$ has the SEP. All terms in \eqref{eq:main2} thus have the SEP and 
therefore \eqref{eq:main2} holds on $X'$, concluding the proof.
\end{proof}

\begin{proposition}\label{Aprop2}
Let $X$ be a reduced complex space of pure dimension $n$. Then $\debar\colon\A_X^{p,q}\to\A_X^{p,q+1}$
and the sheaf complex \eqref{intro:Akplx} is exact.
\end{proposition}
\begin{proof}
Let $\varphi$ be a $\debar$-closed section of $\A_X^{p,q}$ over some small neighborhood $U$ of a given point $x\in X$;
we may assume that $U$ is an analytic subset of some pseudoconvex domain in some $\C^N$. 
As in the proof of Theorem~\ref{main2} above one shows that, for suitable operators $\K$ and $\Proj$, $\varphi=\debar\K\varphi$ if
$q\geq 1$ and $\varphi=\Proj\varphi$ is a section of $\Om^p_X$ if $q=0$.

It remains to see that $\debar\colon\A_X^{p,q}\to\A_X^{p,q+1}$. Let $\varphi$ be a $\debar$-closed section 
of $\A_X^{p,q}$ over some small neighborhood $U$ of a given point $x\in X$; we
may assume that $\varphi$ is of the form \eqref{A} and we will use induction over $\nu$. 
If $\nu=0$ then $\varphi=\xi_0$ is smooth and so
$\debar\varphi$ is in $\E_X^{p,q+1}\subset\A_X^{p,q+1}$. Assume that $\debar\varphi'$ is in $\A_X^{p,*}$
for $\varphi'$ of the form \eqref{A} with $\nu=\ell-1$. Since $\varphi'$ is a section of $\A_X^{p,*}$
it follows from Proposition~\ref{kprop1} that
\begin{equation}\label{fromp}
\varphi'=\debar\K_{\ell}\varphi' + \K_{\ell}(\debar\varphi') + \Proj_{\ell}\varphi'
\end{equation}
as currents on $U'_{reg}$ for some sufficiently small neighborhood $U'$ of $x$, cf.\
the proof of Theorem~\ref{main2} above. As in that same proof \eqref{fromp} extends to hold on $U'$.
The left hand side as well as the last term on the right hand side of \eqref{fromp} are obviously in $\A_X^{p,*}$ 
and since $\debar\varphi'$ is in $\A_X^{p,*}$ by assumption and $\K$-operators preserve $\A_X^{p,*}$ also the second term on the 
right hand side is in $\A_X^{p,*}$. Hence, $\debar\K_{\ell}\varphi'$ is a section of $\A_X^{p,*}$ over $U'$
showing that $\debar\varphi$ is in $\A_X^{p,*}$ for $\varphi$ of the form \eqref{A} with $\nu=\ell$. 
\end{proof}

Notice that Theorem~\ref{main1} follows from Propositions \ref{Aprop1} and \ref{Aprop2}.

\begin{proof}[Proof of Proposition~\ref{serreconditions}]
Assume that condition (i) of Proposition~\ref{serreconditions} holds. Then, in view of the last paragraph in Section~\ref{BHPsheaf},
any holomorphic $p$-form on the regular part at least extends to a section of $\omega^p_X$; in particular, 
such forms are meromorphic. It is thus clear from Proposition~\ref{extensionprop}
that $\Om^p(U)\to \Om^p(U_{reg})$ is surjective for any open $U\subset X$; the injectivity is obvious. We remark that the
implication $\textrm{(i)} \Rightarrow \textrm{(ii)}$ also follows from \cite[Satz~III]{Scheja}.

Assume that condition (ii) of Proposition~\ref{serreconditions} holds. In view of 
\cite[Theorem~1.14, $(d)\Rightarrow (b)$]{SiuTraut} it is sufficient to show that the restriction map
$H^1(U,\Om_X^p)\to H^1(U_{reg},\Om_X^p)$ is injective for any open $U\subset X$. By Corollary~\ref{cor1},
$H^1(U,\Om_X^p)\simeq H^1(\A^{p,\bullet}(U),\debar)$, so let $\varphi\in \A^{p,\bullet}(U)$ be $\debar$-closed
and assume that its image in $H^1(\A^{p,\bullet}(U_{reg}),\debar)$ vanishes, i.e., that there is 
$\psi\in \A^{p,0}(U_{reg})$ such that $\varphi=\debar\psi$ on $U_{reg}$.
Let $x\in U_{sing}$. By Theorem~\ref{main1}, there is a neighborhood $V\subset U$ of $x$ and a $\psi'\in \A^{p,0}(V)$ such that
$\varphi=\debar\psi'$ in $V$. Then $\psi-\psi'$ is holomorphic on $V_{reg}$ and so, by condition (ii), $\psi-\psi'\in \Om^p(V)$.
Hence, $\psi=\psi'+\psi-\psi'$ can be locally extended across $U_{sing}$ to a section of $\A_X^{p,0}$. In view of the 
SEP, extensions are unique and so $\psi\in \A^{p,0}(U)$ 
%\footnote{One could also use that $\A^{p,0}$ is a module over 
%$\E^{0,0}_X$ to patch up local extensions by a partition of unity.} 
and consequently $\debar\psi\in \A^{p,1}(U)$. 
The equality $\varphi=\debar\psi$ on $U_{reg}$ therefore extends to hold on $U$ by the SEP and so $\varphi$ defines the 
zero element in $H^1(U,\Om_X^p)$.
\end{proof}

%We conclude this subsection by showing that if condition (i) of Proposition~\ref{serreconditions} (or equivalently, condition (ii))
%is satisfied and one \emph{a priori} knows that $X$ is locally a complete intersection, then $X$ is smooth. 

\begin{proof}[Proof of Corollary~\ref{smoothcor}]
Assume that $X=\{f_1=\cdots =f_{\kappa}=0\}\subset D\subset\C^N$ 
has codimension $\kappa$ and that $df_1\wedge \cdots\wedge df_{\kappa}\neq 0$ on $X_{reg}$. Let $\tilde{\omega}$
be a meromorphic $n$-form in $D$ such that the polar set of $\tilde{\omega}$ intersects $X$ properly and such that, 
outside of the polar set of $\tilde{\omega}$,
$df_1\wedge\cdots\wedge df_{\kappa}\wedge \tilde{\omega}=dz$ for some local coordinates $z$ in $D$. Let $\omega$ be the pullback
of $\tilde{\omega}$ to $X$. Then $\omega$ is a holomorphic $n$-form on $X_{reg}$ that is uniquely determined by $dz$ and $X$;
in fact, $\omega$ is the Poincar\'{e}-Leray residue of the meromorphic form
$dz/(f_1\cdots f_{\kappa})$. If $\omega$ has a strongly holomorphic extension to $X$, then, since
$df_1\wedge\cdots\wedge df_{\kappa}\wedge \omega=dz$, it follows that 
$df_1\wedge\cdots\wedge df_{\kappa}\neq 0$ on $X$.
\end{proof}

Some \emph{a priori} assumption on $X$ is necessary for Corollary~\ref{smoothcor}. In fact, if 
$X=\{z_1=z_4=0\}\cup \{z_2=z_3=0\}\subset \C^4$ then one can check that any holomorphic $2$-form on $ X_{reg}$ extends
across $X_{sing}$ to a section of $\Om^2_X$.

\subsection{The sheaves $\Bsheaf_X^{n-p,\bullet}$}
To define $\Bsheaf_X^{n-p,q}$ we follow \cite[Definition~4.1]{RSWSerre} and we say that a $(n-p,q)$-current $\psi$
on an open subset $U\subset X$ is a section of $\Bsheaf_X^{n-p,q}$ over $U$ if for every $x\in U$ the germ
$\psi_x$ can be written as a finite sum of terms
\begin{equation}\label{B}
\xi_{\nu} \wedge \check{\K}_{\nu} (\cdots \xi_2\wedge \check{\K}_2 (\xi_1\wedge\check{\K}_1(\omega\wedge\xi_0)) \cdots ),
\end{equation}
where $\omega$ is an $n-p$-structure form and the $\xi_j$ are
smooth $(0,q)$-forms with support where $\zeta\mapsto k_j(\zeta,z)$ is defined. Recall that
$\omega$ is a $(n-p,*)$-current with values in a bundle $\oplus_k E_k\restriction_X$ so we need $\xi_0$
to take values in $\oplus_k E^*_k\restriction_X$ to make $\omega\wedge\xi_0$ scalar-valued.

It is immediate from the definition and from the general properties of the $\check{\K}$-operators
that $\Bsheaf_X^{n-p,q}\subset \W_X^{n-p,q}$, that $\Bsheaf_{X_{reg}}^{n-p,q} = \E_{X_{reg}}^{n-p,q}$,
that the $\check{\K}$-operators and $\check{\Proj}$-operators preserve $\oplus_q\Bsheaf_X^{n-p,q}$,
and that $\oplus_q\Bsheaf_X^{n-p,q}$ is a module over $\oplus_q\E_X^{0,q}$. Let $\psi$ be a 
smooth $(n-p,q)$-form and let $\omega$ be an $n-p$-structure form in a neighborhood of some point in $X$. 
Then, by Lemma~\ref{divlemma}, there is a smooth $(0,q)$-form $\psi'$ (with values in the appropriate bundle)
such that $\psi=\omega_0\wedge\psi'$. Hence we see that $\E_X^{n-p,q}\subset \Bsheaf_X^{n-p,q}$.
Let us also notice that if $\psi$ is in $\Bsheaf_X^{n-p,q}$ then $\debar\psi$ has the SEP. In fact, 
we may assume that $\psi$ is of the form \eqref{B} so that $\psi=\pi_* T\wedge\xi$,
where $T$ is given by \eqref{T}, $\xi$ is a smooth form, and $\pi$ is the natural projection 
$X^{\nu+1}\to X_{z^0}$. Letting $h=h(z^0)$ be a generically non-vanishing holomorphic tuple on $X_{z^0}$,
we have that $\debar\chi(|h|^2/\epsilon)\wedge T\wedge\xi\to 0$ by Lemma~\ref{ASlemma}. Hence, by Lemma~\ref{obslemma},
we see that $\debar\psi$ has the SEP.

\begin{proof}[Proof of Theorem~\ref{main4}]
We first interchange the roles of $p$ and $n-p$ in the formulation of Theorem~\ref{main4}.
Let $D''\Subset D$ be a strictly pseudoconvex neighborhood of $\overline{D}'$
and carry out the construction of $k(\zeta,z)$ and $p(\zeta,z)$ in Section~\ref{intopsection}
in $D''\times D''$ using a weight $g(\zeta,z)$ with respect to $\zeta\in D'$ such that
$\zeta\mapsto g(\zeta,z)$ is holomorphic in $D'$ and $z\mapsto g(\zeta,z)$ has compact support 
in $D''$. Let $\psi\in\Bsheaf^{n-p,q}(X)$. By Proposition~\ref{kprop2} we have 
\begin{equation}\label{kul}
\psi=\debar\check{\K}\psi + \check{\K}(\debar\psi) + \check{\Proj}\psi
\end{equation}
as currents on $X'_{reg}$. From what we noticed just before the proof all terms have the SEP
and so \eqref{kul} holds on $X'$. 
Notice that $\check{\Proj}\psi=A_{q}(\zeta)\wedge\omega_{q}(\zeta)$,
where $\A_{q}$ is holomorphic. Since, if $\Om_X^p$ is Cohen-Macaulay we may
choose $\omega=\omega_0$ to be $\debar$-closed it follows that $\check{\Proj}\psi\in \omega^{n-p}(X')$
if $q=0$ and $\check{\Proj}\psi=0$ if $q\geq 1$.
\end{proof}

\begin{proof}[Proof of Theorem~\ref{main3}]
As in the proof above we interchange the roles of $p$ and $n-p$ in the formulation of Theorem~\ref{main3}.
We have already noted that (i) and (ii) hold. 

To show that $\debar\colon \Bsheaf_X^{n-p,q}\to\Bsheaf_X^{n-p,q+1}$ let $\psi$ be a section of 
$\Bsheaf_X^{n-p,q}$ in a neighborhood of some $x\in X$; we may assume that $\psi$ is of the form \eqref{B}
and we use induction over $\nu$.
If $\nu=0$ then $\psi=\omega\wedge\xi_0$ and it is enough to see that $\debar\omega$ is a section of 
$\Bsheaf_X^{n-p,*}$ (with values in $E\restriction_X$); but since $\debar\omega=f\omega$ this is clear.
The induction step is done in the same way as in the proof of Proposition~\ref{Aprop2}.

To show that $\omega_X^{n-p,q}$ is coherent and that $\omega_X^{n-p}=\omega_X^{n-p,0}$
assume that $X$ can be identified with an analytic subset of a strictly pseudoconvex domain $D\subset \C^N$. 
Recall that \eqref{eq:resol} is a resolution of $\Om_X^p$ in $D$. Taking $\Hom$ into $\mathit{\Omega}^N$ we get a complex 
isomorphic to $(\hol(E^*_{\bullet})\otimes \mathit{\Omega}^N,\debar)$ with associated cohomology sheaves
isomorphic to $\Ext^{\bullet}(\Om_X^p,\mathit{\Omega}^N)$, which are coherent; cf.\ Section~\ref{BHPsheaf}.
We define the map 
\begin{equation*}
\varrho_q\colon \hol(E^*_{\kappa+q})\otimes \mathit{\Omega}^N  \to  \Bsheaf_X^{n-p,q}, \quad
\varrho_q(\xi dz) = i^*\xi \cdot \omega_q.
\end{equation*}
Since
\begin{eqnarray*}
\debar \varrho_q (\xi dz) &=& i^*\xi\cdot \debar \omega_q = i^*\xi\cdot f_{\kappa+q+1}\restriction_X\omega_{q+1}=
i^*f^*_{\kappa+q+1}\restriction_X\xi \cdot \omega_{q+1} \\
 &=& \varrho_{q+1}(f^*_{\kappa+q+1} \xi dz) 
\end{eqnarray*}
the map $\varrho_{\bullet}$ is a map of complexes and so induces a map on cohomology.
In view of Proposition~\ref{BHPprop} the proof will be complete if we show that $\varrho_{\bullet}$ is a 
quasi-isomorphism.

Since $i_*\omega_q=R_{\kappa+q}\wedge dz$ it follows from \cite[Theorem~7.1]{ANoetherDual} that the map on cohomology is injective.
For the surjectivity, let $\psi\in \Bsheaf^{n-p,q}(X)$ be $\debar$-closed and choose a weight $g(\zeta,z)$ 
in the kernels $k(\zeta)$ and $p(\zeta,z)$ with respect 
to $\zeta$ in some $D'\Subset D$ such that $\zeta\mapsto g(\zeta,z)$ is holomorphic in $D'$ and $z\mapsto g(\zeta,z)$ has
compact support in $D$. As in the proof of Theorem~\ref{main4} we get that 
$\psi=\debar\check{\K}\psi+\check{\Proj}\psi$ on $X'_{reg}:=X_{reg}\cap D'$ and so the cohomology class of $\psi$ is represented
by $\check{\Proj}\psi$.
From the definition of $p(\zeta,z)$ in Section~\ref{intopsection} we see that
\begin{equation*}
\check{\Proj}\psi(\zeta)=\pm\omega_q(\zeta)\wedge \int_{X_z} \tilde{p}_{\kappa+q}(\zeta,z)\wedge \psi(z)
\end{equation*}
and $\zeta\mapsto \tilde{p}_{\kappa+q}(\zeta,z)$ is a section of $\hol(E_{\kappa+q}^*)$ over $D'$ by the choice of $g$.
We finally show that 
\begin{equation}\label{sent3}
f^*_{\kappa+q+1} \int_{X_z} \tilde{p}_{\kappa+q}(\zeta,z)\wedge \psi(z)=0.
\end{equation}
First notice that it follows from \eqref{sent1} and \eqref{sent2} that, for each $k$, 
$\tilde{p}_k(\zeta,z)\wedge d\eta=H^0_k\wedge g_{N-k}$. Moreover,
\begin{eqnarray*}
f^*_{k+1}H^0_k\wedge g_{N-k} &=& H^0_kf_{k+1} \wedge g_{N-k} = \big(f_1(z)H^1_{k+1} +\delta_{\eta}H^0_{k+1}\big)\wedge g_{N-k} \\
&=& f_1(z)H^1_{k+1}\wedge g_{N-k} \pm H^0_{k+1}\wedge \delta_{\eta}g_{N-k} \\ 
&=& f_1(z)H^1_{k+1}\wedge g_{N-k} \pm H^0_{k+1}\wedge \debar g_{N-k-1} \\
&=& f_1(z)H^1_{k+1}\wedge g_{N-k} +\debar (H^0_{k+1}\wedge g_{N-k-1}) \\
&=:& (f_1(z)A_k + \debar B_k)\wedge d\eta,
\end{eqnarray*}
where $A_k$ and $B_k$ take values in $\textrm{Hom}(E^{\zeta}_{k+1},E_1^z)$ and 
$\textrm{Hom}(E^{\zeta}_{k+1},E_0^z)$ respectively; the second equality follows from the properties of
the Hefer morphisms, the third by noting that 
$0=\delta_{\eta}(H^0_{k+1}\wedge g_{N-k})=\delta_{\eta}H^0_{k+1}\wedge g_{N-k}\pm H^0_{k+1}\wedge \delta_{\eta}g_{N-k}$,
the fourth since $g$ is a weight, the fifth since the Hefer morphisms are holomorphic, and the sixth by collecting 
all $d\eta_j$. Hence, we get that $f^*_{k+1}\tilde{p}_k(\zeta,z)=f_1(z)A_k + \debar B_k$.
Since $f_{1\restriction_X}=0$ and by Stokes' theorem, \eqref{sent3} follows.
\end{proof}

%%%%%%%%%%%%%%%%%%%%%%%%%%%%%%%%%%%%%%%%%%%%%%%%%%%%%%%%%%%%%%%%%%%%%%%%%%%%%%%%%%%%%%%%%%%%%%%%%%%%%%%%%%%%%%%%%%%%%%%%%%%%%%%%%%%

\section{Serre duality}\label{Serresection}

\subsection{The trace map}
The key to define the trace map is the following slight generalization of \cite[Theorem~5.1]{RSWSerre};
the proof of that theorem goes through in our case essentially verbatim. 
% Notice however
% that in \emph{ibid}.\ the notation $\A_X^{n,*}$ is used in place of our $\Bsheaf_X^{n,*}$.

\begin{theorem}\label{RSWthm}
Let $X$ be a reduced complex space of pure dimension $n$. There is a unique map
\begin{equation*}
\wedge \colon \A_X^{p,q}\times \Bsheaf_X^{n-p,q'} \to \W_X^{n,q+q'}
\end{equation*}
extending the exterior product on $X_{reg}$. Moreover,
if $\varphi$ and $\psi$ are sections of $\A_X^{p,q}$ and $\Bsheaf_X^{n-p,q'}$, respectively,
then $\debar (\varphi\wedge\psi)$ has the SEP.
\end{theorem}

It follows that $\debar (\varphi\wedge\psi)=\debar\varphi\wedge\psi+(-1)^{p+q}\varphi\wedge\debar\psi$
since both sides have the SEP and it clearly holds on $X_{reg}$.

Let $\varphi\in \A^{p,q}(X)$ and $\psi\in\Bsheaf^{n-p,n-q}(X)$ and assume that at least one of
$\varphi$ and $\psi$ has compact support. By Theorem~\ref{RSWthm}, $\varphi\wedge\psi$ is a
well-defined section of $\W_X^{n,n}$ with compact support and we may define the trace map 
$(\varphi,\psi)\mapsto \int_X \varphi\wedge\psi$; the integral is interpreted as
the action of $\varphi\wedge\psi$ on the constant function $1$ on $X$.
We notice that if $h$ is a generically non-vanishing holomorphic 
section of a Hermitian vector bundle such that $\{h=0\}\supset X_{sing}$ then the trace map may be computed as
$\lim_{\epsilon\to 0}\int_X \chi(|h|^2/\epsilon)\, \varphi\wedge \psi$.
We get an induced trace map on the level of cohomology since if, say, 
$\varphi=\debar\tilde{\varphi}$ for some 
$\tilde{\varphi}\in \A^{p,q-1}(X)$ with compact support if $\varphi$ has, then
$\varphi\wedge\psi=\debar(\tilde{\varphi}\wedge\psi)$ by the Leibniz rule 
and so $\int_X\varphi\wedge\psi=0$.

\subsection{Local duality}
Let $\tilde{X}$ be an analytic subset of $\overline{D}\subset\C^N$, where $D$ is pseudoconvex, and set 
$X:=\tilde{X}\cap D$. Let $F$ be a holomorphic vector bundle on $X$ and let $\F$ be the associated 
locally free $\hol_X$-module. Since $X$ is Stein and $\F\otimes \Om_X^p$ is coherent it follows from Corollary~\ref{cor1}
that the complex
\begin{equation*}
0\to \A^{p,0}(X,F) \stackrel{\debar}{\longrightarrow} \A^{p,1}(X,F) \stackrel{\debar}{\longrightarrow}
\cdots \stackrel{\debar}{\longrightarrow} \A^{p,n}(X,F) \to 0
\end{equation*}
is exact except for on the level $0$ where the cohomology is $\Om^p(X,F)$.
We endow $\Om^p(X,F)$ with the standard canonical Fr\'{e}chet space topology, see, e.g.,
\cite[Chapter~IX]{Demailly}.

\begin{theorem}\label{localdual}
Let $\Bsheaf_c^{n-p,q}(X,F^*)$ be the space of sections of $\F^*\otimes\Bsheaf_X^{n-p,q}$ with compact support in $X$.
The complex
\begin{equation}\label{hund}
0\to \Bsheaf_c^{n-p,0}(X,F^*) \stackrel{\debar}{\longrightarrow} \Bsheaf_c^{n-p,1}(X,F^*) \stackrel{\debar}{\longrightarrow}
\cdots \stackrel{\debar}{\longrightarrow} \Bsheaf_c^{n-p,n}(X,F^*) \to 0
\end{equation}
is exact except for on the level $n$ and the pairing
\begin{equation}\label{hund2}
\Om^p(X,F)\times H^n\big(\Bsheaf_c^{n-p,\bullet}(X,F^*),\debar\big) \to  \C, \quad
(\varphi,[\psi])  \mapsto  \int_X \varphi\cdot \psi
\end{equation}
makes $H^n\big(\Bsheaf_c^{n-p,\bullet}(X,F^*),\debar\big)$ the topological dual of $\Om^p(X,F)$.
\end{theorem}

\begin{proof}[Sketch of proof.]
Since we are in the local situation we may assume that an element in $\Bsheaf^{n-p,q}_c(X,F^*)$
is just a tuple of elements in $\Bsheaf_c^{n-p,q}(X)$ and carry out the following argument component-wise.
Let $\psi\in \Bsheaf_c^{n-p,q}(X)$ be $\debar$-closed. Let $D'\Subset D''\subset D$, where
$\textrm{supp}\,\psi\subset D'$ and $D''$ is strictly pseudoconvex, and construct $k(\zeta,z)$ and 
$p(\zeta,z)$ as in Section~\ref{intopsection} with a weight $g(\zeta,z)$ with respect to $z\in D'$ such 
that $z\mapsto g(\zeta,z)$ is holomorphic in $D'$ and $\zeta\mapsto g(\zeta,z)$ has compact support in
$D''$. Then $p(\zeta,z)=\sum_k\tilde{p}_{\kappa+k}(\zeta,z)\wedge \omega_k(\zeta)$,
where $\zeta\mapsto \tilde{p}_{\kappa+k}(\zeta,z)$ has compact support in $D''$ and 
$z\mapsto \tilde{p}_{\kappa+k}(\zeta,z)$ is a section of $\Om^p_X$ over $X':=X\cap D'$.

As in the proof of Theorem~\ref{main4} we get 
$\psi=\debar \check{\K}\psi + \check{\Proj}\psi$
in $X'$. From the properties of $p(\zeta,z)$ we get that $\check{\Proj}\psi=0$ if $q<n$
so \eqref{hund} is exact except for on the level $n$.
If $q=n$ then the cohomology class of $\psi$ is represented by $\check{\Proj}\psi$ and 
\begin{equation*}
\check{\Proj}\psi=\pm \sum_{k\geq 0} \omega_k(\zeta)\wedge
\int_{X_z} \tilde{p}_{\kappa+k}(\zeta,z)\wedge\psi(z).
\end{equation*}
Hence, if $\int_X\varphi \psi=0$ for all $\varphi\in \Om^p(X)$ then $\check{\Proj}\psi =0$
and the cohomology class of $\psi$ thus is $0$. It follows that 
$H^n(\Bsheaf_c^{n-p,\bullet}(X),\debar)$, via \eqref{hund2}, is a 
subset of the topological dual of $\Om^p(X)$.

Let $\lambda$ be a continuous linear functional on $\Om^p(X)$. Then $\lambda$ induces 
a continuous functional $\tilde{\lambda}$ on $\mathit{\Omega}^p(D)$ that has to be carried by some 
compact $K\Subset D$. By the Hahn-Banach theorem there is an $(N-p,N)$-current $\mu$ of order $0$ 
in $D$ with support in a neighborhood $U(K)\Subset D$ of $K$ such that 
$\tilde{\lambda}(\tilde{f})=\int \tilde{f}\wedge \mu$ for all $\tilde{f}\in \mathit{\Omega}^p(D)$.
Now choose a weight $g(\zeta,z)$ with respect to $z\in U(K)$ that is holomorphic for $z\in U(K)$ and has 
compact support in $D_{\zeta}$ and let $p(\zeta,z)=\sum_k\tilde{p}_{\kappa+k}(\zeta,z)\wedge \omega_k(\zeta)$
be a corresponding integral kernel. We set
\begin{equation*}
\check{\Proj}\mu:=\sum_{k\geq 0}  \omega_k(\zeta)\wedge
\int_{D_z} \tilde{p}_{\kappa+k}(\zeta,z)\wedge\mu(z)
\end{equation*}
and observe that $\check{\Proj}\mu\in\Bsheaf_c^{n-p,n}(X)$.
Let $\varphi\in\Om^p(X)$ and set $\tilde{\varphi}:=\Proj \varphi$. Then $\tilde{\varphi}\in \mathit{\Omega}^p(U(K))$
by the choice of weight and moreover, $\tilde{\varphi}_{\restriction_{U(K)\cap X}}=\varphi_{\restriction_{U(K)\cap X}}$.
We get
\begin{equation*}
\lambda(\varphi)=\tilde{\lambda}(\tilde{\varphi}) = \int_{D_z} \tilde{\varphi}\wedge\mu
=\int_{D_z} \Proj\varphi\wedge\mu =\int_{X_{\zeta}} \varphi\wedge \check{\Proj}\mu
\end{equation*}
and so $\lambda$ is given by integration against $\check{\Proj}\mu\in\Bsheaf_c^{n-p,n}(X)$.
For more details of the last part of the proof see the proof of \cite[Theorem~6.1]{RSWSerre}.
\end{proof}

\subsection{Global duality}
Let us briefly recall how one can patch up the local duality to the global one of Theorem~\ref{main5};
cf., e.g., \cite[Section~6.2]{RSWSerre}.
Let $\mathcal{U}:=\{U_j\}$ be a locally finite open covering of $X$ such that each $U_j$ can be identified with
an analytic subset of some pseudoconvex domain in some $\C^N$. In view of Theorem~\ref{main1}
and Corollary~\ref{cor1} this gives us a Leray covering for $\F\otimes\Om_X^p$. Recall that 
spaces of sections of $\F\otimes\Om_X^p$ has a standard Fr\'{e}chet space structure.
We let $C^k(\mathcal{U}, \F\otimes\Om_X^p)$ be the group of formal sums
\begin{equation*}
\sum_{i_0 \cdots i_k} \varphi_{i_0\cdots i_k} U_{i_0}\wedge \cdots \wedge U_{i_k}, \quad  \varphi_{i_0\cdots i_k}
\in \F\otimes\Om^p(U_{i_0}\cap \cdots \cap U_{i_k}),
\end{equation*}
with the product topology; $U_{i_0}\wedge \cdots \wedge U_{i_k}$ is the formal exterior product 
of the symbols $U_i$ with the suggestive formal computation rules, e.g., $U_1\wedge U_2=-U_2\wedge U_1$.
Each element of $C^k(\mathcal{U}, \F\otimes \Om_X^p)$ thus has a unique representation of the form
$\sum_{i_0 < \cdots < i_k} \varphi_{i_0\cdots i_k} U_{i_0}\wedge \cdots \wedge U_{i_k}$
that we will abbreviate as $\sum'_{|I|=k+1} \varphi_IU_I$.
We define a coboundary operator $\delta\colon C^k(\mathcal{U}, \F\otimes \Om_X^p)\to C^{k+1}(\mathcal{U}, \F\otimes \Om_X^p)$
by
\begin{equation*}
\delta \sum'_{|I|=k+1} \varphi_IU_I := \sum'_{|I|=k+1} \varphi_IU_I \wedge \sum_j U_j =
\sum'_{|I|=k+1}\sum_j \varphi_I \restriction_{U_I\cap U_j} U_I\wedge U_j,
\end{equation*}
which is continuous, and we get the following complex of Fr\'{e}chet spaces
\begin{equation}\label{checkcplx}
0\to C^0(\mathcal{U},\F\otimes\Om_X^p) \stackrel{\delta}{\longrightarrow} 
C^1(\mathcal{U},\F\otimes\Om_X^p) \stackrel{\delta}{\longrightarrow} \cdots.
\end{equation}
The $q^{\textrm{th}}$ cohomology group of this complex is isomorphic to 
$H^q(X,\F\otimes\Om_X^p)$ and in fact, the standard topology on $H^q(X,\F\otimes\Om_X^p)$
is defined so that the isomorphism also is a homeomorphism.

\smallskip

Let $B^{n-p}$ be the precosheaf (see, e.g., \cite[Section~3]{AK}) defined by assigning
to each open $U\subset X$ the space 
$B^{n-p}(U):=H^n\big(\Bsheaf_c^{n-p,\bullet}(U,F^*),\debar\big)$
and for $U'\subset U$ the inclusion map $i^{U'}_U \colon B^{n-p}(U')\to B^{n-p}(U)$ given by extension by $0$.
We let, for $k\geq 0$, $C_c^{-k}(\mathcal{U},B^{n-p})$ be the group of formal sums
\begin{equation*}
\sum_{i_0 \cdots i_k} [\psi_{i_0\cdots i_k}]_{\debar} U^*_{i_0}\wedge\cdots\wedge U^*_{i_k}, \quad
\psi_{i_0\cdots i_k}\in \Bsheaf_c^{n-p,n}(U_{i_0}\cap\cdots\cap U_{i_k},F^*),
\end{equation*}
with the suggestive computation properties and only finitely many $[\psi_{i_0\cdots i_k}]_{\debar}$ non-zero.
We define the coboundary operator
$\delta^*\colon C_c^{-k}(\mathcal{U},B^{n-p})\to C_c^{-k+1}(\mathcal{U},B^{n-p})$
by
\begin{equation*}
\delta^* \sum'_{|I|=k+1} [\psi_I] U^*_I := \sum_j U_j\lrcorner \sum'_{|I|=k+1} [\psi_I] U^*_I
=\sum'_{|I|=k+1}\sum_j i^{U_I}_{U_{I\setminus \{j\}}} [\psi_I]U_j\lrcorner U^*_I,
\end{equation*}
where $\lrcorner$ is formal interior multiplication. We get the complex
\begin{equation}\label{cocheckcplx}
0 \leftarrow C_c^{0}(\mathcal{U},B^{n-p}) \stackrel{\delta^*}{\longleftarrow} 
C_c^{-1}(\mathcal{U},B^{n-p}) \stackrel{\delta^*}{\longleftarrow} \cdots.
\end{equation}
By Theorem~\ref{localdual}, $C_c^{-k}(\mathcal{U},B^{n-p})$ is the topological dual of 
$C^k(\mathcal{U},\F\otimes\Om_X^p)$ via the pairing
$C^k(\mathcal{U},\F\otimes\Om_X^p) \times C_c^{-k}(\mathcal{U},B^{n-p}) \to  \C$ given by
\begin{equation}\label{vaccin2}
(\varphi,[\psi]_{\debar})=\big(\sum'_{|I|=k+1}\varphi_I U_I, \sum'_{|I|=k+1}[\psi_I]_{\debar} U^*_I\big)  \mapsto 
\int_X \varphi\lrcorner \psi = \sum'_{|I|=k+1} \int_X \varphi_I\wedge\psi_I.
\end{equation}
Moreover, if $\varphi \in C^{k-1}(\mathcal{U},\F\otimes\Om_X^p)$ and $[\psi]\in C_c^{-k}(\mathcal{U},B^{n-p})$, then
\begin{equation*}
\int_X  \varphi\lrcorner \delta^* \psi = \int_X \varphi \lrcorner (\sum_jU_j\lrcorner \psi)
=\int_X (\varphi\wedge\sum_j U_j)\lrcorner \psi = \int_X \delta \varphi \lrcorner \psi
\end{equation*}
and so \eqref{cocheckcplx} is the dual complex of \eqref{checkcplx}. It follows, see, e.g., \cite[Lemme~2]{RaRu},
that
\begin{equation}\label{lampa1}
\textrm{Ker}\big(\delta^* \colon C^{-q}_c(\mathcal{U},B^{n-p})\to C^{-q+1}_c(\mathcal{U},B^{n-p})\big)/
\overline{\delta^* C^{-q-1}_c(\mathcal{U},B^{n-p})}
\end{equation}
is the topological dual of 
\begin{equation}\label{lampa2}
\textrm{Ker}\big(\delta \colon C^{q}(\mathcal{U},\F\otimes\Om_X^p)\to C^{q+1}(\mathcal{U},\F\otimes\Om_X^p)\big)/
\overline{\delta C^{q-1}(\mathcal{U},\F\otimes\Om_X^p)}.
\end{equation}
Now, if $H^q(X,\F\otimes\Om_X^p)$ and $H^{q+1}(X,\F\otimes\Om_X^p)$ are Hausdorff, then the closure signs in
\eqref{lampa1} and \eqref{lampa2} are superfluous and so $H^{-q}(C^{\bullet}_c(\mathcal{U},B^{n-p}),\delta^*)$ is the 
topological dual of $H^q(X,\F\otimes\Om_X^p)$ in this case, via the pairing induced by \eqref{vaccin2}. 

To understand $H^{-q}(C^{\bullet}_c(\mathcal{U},B^{n-p}), \delta^*)$, consider the double complex 
\begin{equation*}
K^{-i,j}:=C_c^{-i}(\mathcal{U},\Bsheaf_c^{n-p,j}),
\end{equation*} 
where $\Bsheaf_c^{n-p,j}$ is the precosheaf
$U\mapsto \Bsheaf_c^{n-p,j}(U)$ with inclusion maps given by extending by $0$, the map 
$K^{-i,j}\to K^{-i+1,j}$ is $\delta^*$, and the map $K^{-i,j}\to K^{-i,j+1}$ is $\debar$.
For each $i\geq 0$ the ``row'' $K^{-i,\bullet}$ is, by Theorem~\ref{localdual}, exact except for on 
the level $n$ where the cohomology is $C^{-i}_c(\mathcal{U},B^{n-p})$. 
Since the $\Bsheaf_X$-sheaves are fine it follows from, e.g., \cite[Lemma~6.2]{RSWSerre} that,
for each $j\geq 0$, the ``column'' $K^{\bullet,j}$ is exact except for on the level $0$ where the
cohomology is $\Bsheaf_c^{n-p,j}(X,F^*)$. From, e.g., a spectral sequence argument it thus follows that 
\begin{equation}\label{vaccin}
H^{-q}\big(C^{\bullet}_c(\mathcal{U},B^{n-p}),\delta^*\big) \simeq
H^{n-q}\big(\Bsheaf_c^{n-p,\bullet}(X,F^*), \debar\big).
\end{equation}
Hence, we have a non-degenerate pairing \eqref{par1} but we have not proved that it is given by \eqref{par2}. 
To do this one makes the isomorphisms $H^q(\A^{p,\bullet}(X,F),\debar)\simeq H^q(X,\F\otimes\Om_X^p)$ and
\eqref{vaccin} explicit; see the proof of \cite[Theorem~1.3]{RSWSerre} for details.

\end{document}